\documentclass[a4paper, 11pt]{amsart}   
\usepackage{amssymb,amscd,latexsym}   
\usepackage{amsmath}
\usepackage{amsthm}
\usepackage{mathdots}
\usepackage[pagebackref,colorlinks=true,linkcolor=blue,urlcolor=blue]{hyperref}
\usepackage{color}
\usepackage{tabularx}
\usepackage{amsfonts}
\usepackage{paralist}
\usepackage{aliascnt}
\usepackage[initials]{amsrefs}
\usepackage{amscd}
\usepackage{blkarray}
\usepackage{mathbbol}
\usepackage{setspace}
\usepackage[inner=2.5cm,outer=2.5cm, bottom=3.3cm]{geometry}
\usepackage{tikz, tikz-cd}

\usetikzlibrary{matrix}
\usetikzlibrary{arrows}

\BibSpec{collection.article}{%
	+{}  {\PrintAuthors}                {author}
	+{,} { \textit}                     {title}
	+{.} { }                            {part}
	+{:} { \textit}                     {subtitle}
	+{,} { \PrintContributions}         {contribution}
	+{,} { \PrintConference}            {conference}
	+{}  {\PrintBook}                   {book}
	+{,} { }                            {booktitle}
	+{,} { }                            {series}
	+{, vol.} { }                            {volume}
	+{,} { }                            {publisher}
	+{,} { \PrintDateB}                 {date}
	+{,} { pp.~}                        {pages}
	+{,} { }                            {status}
	+{,} { \PrintDOI}                   {doi}
	+{,} { available at \eprint}        {eprint}
	+{}  { \parenthesize}               {language}
	+{}  { \PrintTranslation}           {translation}
	+{;} { \PrintReprint}               {reprint}
	+{.} { }                            {note}
	+{.} {}                             {transition}
	+{}  {\SentenceSpace \PrintReviews} {review}
}

\newcommand{\KK}{\mathbb{K}}

\newcommand{\NN}{\normalfont\mathbb{N}}
\newcommand{\ZZ}{{\normalfont\mathbb{Z}}}

\newcommand{\pp}{{\normalfont\mathfrak{p}}}
\newcommand{\aaa}{\normalfont\mathfrak{a}}
\newcommand{\bbb}{\normalfont\mathfrak{b}}
\newcommand{\qqq}{\mathfrak{q}}

\newcommand{\Ker}{\normalfont\text{Ker}}

\newcommand{\Quot}{\normalfont\text{Quot}}

\newcommand{\Ass}{\normalfont\text{Ass}}

\newcommand{\Hom}{\normalfont\text{Hom}}

\newcommand{\leng}{{\normalfont\text{length}}}

\newcommand{\BBB}{\mathfrak{B}}

\newcommand{\LL}{\mathbb{L}}

\newcommand{\FF}{\mathbb{F}}

\newcommand{\HH}{\normalfont\text{H}}

\newcommand{\Sol}{\normalfont\text{Sol}}

\newcommand{\AAA}{\mathfrak{A}}

\newcommand{\Spec}{{\normalfont\text{Spec}}}

\newcommand{\Diff}{{\normalfont\text{Diff}}}
\newcommand{\MM}{\mathcal{M}}

\newcommand{\amult}{{\normalfont \text{amult}}}
\newcommand{\DiffR}{\Diff_{R/\KK}}
\newcommand{\Diag}{\Delta_{R/\KK}}

\newtheorem{theorem}{Theorem}[section]

\newaliascnt{headcor}{headthm}

\aliascntresetthe{headcor}

\newaliascnt{headconj}{headthm}

\aliascntresetthe{headconj}

\newaliascnt{corollary}{theorem}
\newtheorem{corollary}[corollary]{Corollary}
\aliascntresetthe{corollary}

\newaliascnt{lemma}{theorem}
\newtheorem{lemma}[lemma]{Lemma}
\aliascntresetthe{lemma}

\newaliascnt{conjecture}{theorem}

\aliascntresetthe{conjecture}

\newaliascnt{proposition}{theorem}
\newtheorem{proposition}[proposition]{Proposition}
\aliascntresetthe{proposition}

\theoremstyle{definition}
\newaliascnt{definition}{theorem}
\newtheorem{definition}[definition]{Definition}
\aliascntresetthe{definition}

\newaliascnt{notation}{theorem}
\newtheorem{notation}[notation]{Notation}
\aliascntresetthe{notation}

\newaliascnt{example}{theorem}
\newtheorem{example}[example]{Example}
\aliascntresetthe{example}

\newaliascnt{examples}{theorem}

\aliascntresetthe{examples}

\newaliascnt{remark}{theorem}
\newtheorem{remark}[remark]{Remark}
\aliascntresetthe{remark}

\newaliascnt{problem}{theorem}

\aliascntresetthe{problem}

\newaliascnt{construction}{theorem}

\aliascntresetthe{construction}

\newaliascnt{setup}{theorem}
\newtheorem{setup}[setup]{Setup}
\aliascntresetthe{setup}

\newaliascnt{algorithm}{theorem}
\newtheorem{algorithm}[algorithm]{Algorithm}
\aliascntresetthe{algorithm}

\newaliascnt{observation}{theorem}

\aliascntresetthe{observation}

\newaliascnt{defprop}{theorem}

\aliascntresetthe{defprop}

\def\equationautorefname~#1\null{(#1)\null}
\def\sectionautorefname~#1\null{Section #1\null}
\def\subsectionautorefname~#1\null{\S #1\null}

\begin{document}

\title{Primary Decomposition with  Differential Operators}
\author{Yairon Cid-Ruiz}
\address[Cid-Ruiz]{Department of Mathematics: Algebra and Geometry, Ghent University, Belgium}
\email{Yairon.CidRuiz@UGent.be}
\author{Bernd Sturmfels}
\address[Sturmfels]{MPI-MiS Leipzig and UC Berkeley}
\email{bernd@mis.mpg.de}

\begin{abstract}
	We introduce differential primary decompositions for ideals in a commutative ring.
	Ideal membership is characterized by differential conditions. The minimal number
	of conditions needed is  the arithmetic multiplicity.
	Minimal differential primary decompositions are unique up to change of bases.
		Our results generalize the construction of Noetherian operators
	for primary ideals in the analytic theory of Ehrenpreis-Palamodov,
	and they offer a concise method for representing affine schemes.
	The case of modules is also addressed.
		We implemented an algorithm in {\tt Macaulay2} that computes the 
	minimal decomposition for an ideal in a polynomial~ring.
\end{abstract}

\maketitle

\section{Introduction}

Macaulay's theory of inverse systems \cite{Groebner} employs differential operators
to characterize membership in an ideal that is primary to the maximal ideal in a power series ring or to the maximal irrelevant ideal in a polynomial ring.
The number of operators needed is the multiplicity of the ideal. 
Building on \cite{BRUMFIEL_DIFF_PRIM},
this description was extended
to primary ideals in our previous papers \cite{NOETH_OPS, PRIM_IDEALS_DIFF_EQS}. 
The present article develops a minimal such representation
for arbitrary ideals in a commutative $\KK$-algebra that is essentially of finite type over a perfect field $\KK$. 
We introduce \emph{differential primary decompositions}.
These are  differential conditions that  characterize ideal membership.

\begin{example} \label{ex:triplepalamodov} We describe an ideal $I$ in $R = \mathbb{Q}[x,y,z]$.
	A polynomial $f$ lies in $I$ if  and only~if
	\begin{itemize}
		\item[(a)] both $f$ and $x \frac{\partial f }{\partial y} + \frac{\partial f}{\partial z}$ vanish on the $x$-axis, \smallskip
		\item[(b)] both $f$ and $y \frac{\partial f }{\partial z} + \frac{\partial f}{\partial x}$ vanish on the $y$-axis, \smallskip
		\item[(c)] both $f$ and $z \frac{\partial f }{\partial x} + \frac{\partial f}{\partial y}$ vanish on the $z$-axis, and \smallskip 
		\item[(d)] $ \frac{\partial^3 f}{\partial x \partial y \partial z} + 
		\frac{\partial^3 f}{\partial x^2 \partial y} +
		\frac{\partial^3 f}{\partial y^2 \partial z} +
		\frac{\partial^3 f}{\partial z^2 \partial x} $
		vanishes at the origin $(0,0,0)$. \smallskip
	\end{itemize}
	More familiar formats would be a list of ideal generators or a minimal primary decomposition:
	\begin{equation*}
		\label{eq:Igens}
		\begin{matrix} \quad
			I &  = &  \langle \,x y z^2, \,xy^2z, \,x^2 y z,\,
			y^2z^2, \,2 x yz -xz^2+ yz^3, 2 x y z  -x^2 y+ x^3 z, 2xyz-y^2z + xy^3 \rangle \smallskip \\
& = &		\langle\, y^2,\,z^2,\,y-xz \,\rangle \,\cap\,
			\langle \,x^2, \,y^2 z , \,z- xy\, \rangle \,\cap \,
			\langle \,x^2, \,y^2, \,x - yz \,\rangle \\ & &  \, \, \qquad \cap \,\,
			\langle \,x^3,\,y^3,\,z^3,\,xy^2,\,yz^2,\,zx^2,\,2xyz-x^2y,\, 2xyz-y^2z, \,2xyz-z^2x \,\rangle. 
		\end{matrix}
	\end{equation*}
We notice  four associated primes:
	$\pp_1 = \langle y,z \rangle$,
	$\pp_2 = \langle x,z \rangle$,
	$\pp_3 = \langle x,y \rangle$ and $\pp_4 = \langle x,y,z \rangle$.
In (a)-(d) we characterized membership in $I$ by a set 
	of linear differential operators for each prime:
		$$ \AAA_1 = \{ 1,\, x \partial_y {+} \partial_z \},\,
	 \AAA_2  = \{ 1,\, y \partial_z {+} \partial_x \,\},\,
	 \AAA_3 = \{  1,\, z \partial_x {+} \partial_y \}, 
	 \AAA_4 = \{\partial_x \partial_y \partial_z {+} 
			\partial_x^2 \partial_y {+} \partial_y^2 \partial_z {+}\partial_z^2 \partial_x \}. $$
The primary ideal $\langle\, y^2,\,z^2,\,y-xz \,\rangle$ is a famous example due to
Palamodov \cite{PALAMODOV}. He showed that
membership in this ideal cannot be described by differential operators 
	with constant coefficients.  
\end{example}

We now discuss the issues that are addressed in this paper. 
 Consider an ideal $I $ in an essentially of finite type $\KK$-algebra $ R$.
 Suppose that its set of associated primes is
$\Ass(R/I) = \{\pp_1,\ldots,\pp_k\} \subset \Spec(R)$.                                      
We wish to characterize ideal membership in $I$ by means of
differential operators. 
The natural place to look for these  is the
ring ${\rm Diff}_{R/\KK}(R,R)$ of  $\KK$-linear differential operators on $R$.

As in  \autoref{ex:triplepalamodov}, we hope to find
finite subsets $\AAA_1,\ldots,\AAA_k \subset {\rm Diff}_{R/\KK}(R,R)$ such that
\begin{equation}
\label{eq:diffprimdec}
I \,\,=\,\, \big\lbrace f \in R \mid \delta(f) \in \pp_i \text{ for  all }  \delta \in \AAA_i
\text{ and } i=1,2,\ldots,k \big\rbrace .
\end{equation}
Such differential primary decompositions exist for ideals $I$
in a polynomial ring $R = \KK[x_1,\ldots,x_n]$.
An instance with $n=3$, $k=4$, $|\AAA_1| = |\AAA_2| = | \AAA_3| = 2$ and $| \AAA_4| = 1$
was shown in \autoref{ex:triplepalamodov}.
This existence result is the Fundamental Principle of Palamodov-Ehrenpreis,
which is important in analysis \cite{Bjoerk, Ehrenpreis, Hoermander, PALAMODOV}.
The elements $\delta \in \AAA_i$ are known as Noetherian operators. They are the
key to solving linear partial differential equations with constant coefficients.
The computation~of Noetherian operators was addressed in
\cite{BRUMFIEL_DIFF_PRIM,  CCHKL, CHKL,
NOETH_OPS, PRIM_IDEALS_DIFF_EQS, OBERST_NOETH_OPS} and
continues to be of interest for both PDE and computer algebra.
For recent work along these lines see \cite{aitelmanssour_harkonen_sturmfels_2021, HHS}.
If each $\pp_i$ is a rational maximal ideal then we are dealing with inverse systems \cite{Groebner} 
and identifying the $\AAA_i$ is standard textbook material \cite[Theorem 3.27]{Mateusz}.
An algorithm for the case when
$R$ is a polynomial ring and $I$ is primary was given recently in \cite{PRIM_IDEALS_DIFF_EQS},
where  ${\rm Diff}_{R/\KK}(R,R)$ is the Weyl algebra $\KK\langle x_1,\ldots,x_n,\partial_{x_1},\ldots,\partial_{x_n} \rangle$,
and punctual Hilbert schemes play a major role.

\smallskip 

In this article, we present answers to the  following two questions:
\begin{itemize}
\item[(a)] \label{problemA} What is the analog of the representation \autoref{eq:diffprimdec} in $\KK$-algebras other than polynomial rings?
\item[(b)] \label{problemB} If a representation \autoref{eq:diffprimdec} exists for an ideal $I \subset R$, what is the minimal size of the sets $\AAA_i$?
\end{itemize}

To motivate \hyperref[problemA]{Problem~(a)}, we note that \autoref{eq:diffprimdec} fails for $\KK$-algebras $R$
that are not regular. 
The noncommutative ring ${\rm Diff}_{R/\KK}(R,R)$ can be very complicated.
It is usually not Noetherian.  A classical example is the cubic cone that was studied by
Bern\v{s}te\u{\i}n, Gel{\cprime}fand and Gel{\cprime}fand in
\cite{BGG_NON_NOETHERIAN_DIFF}:
\begin{equation}
\label{eq:Rnonsmooth} R\,=\, \mathbb{C}[x,y,z]/\langle x^3 + y^3 + z^3 \rangle.
\end{equation}
If $I$ is a power of the maximal ideal $\langle x,y,z \rangle$ in $R$ then
a representation \autoref{eq:diffprimdec} does not exist. This follows from
\cite[Example 5.2]{NOETH_OPS}. The issue is that there are too few differential operators on $R$.

\hyperref[problemB]{Problem~(b)} is motivated by a foundational question in computational algebraic geometry:
how to measure the complexity of a subscheme in affine space or projective space?
Our answer is drawn from \cite{STV_DEGREE}. For any $\pp_i \in {\rm Ass}(R/I)$ let
 ${\rm mult}_I(\pp_i)$ denote the length of the 
largest ideal of finite length in $R_{\pp_i}/I R_{\pp_i}$.
We define the {\em arithmetic multiplicity} of the ideal $I$ to be the sum 
\begin{equation}
\label{eq:arithmeticmultiplicity}
{\rm amult}(I) \,\,\, = \,\,\,   {\rm mult}_I(\pp_1) + {\rm mult}_I(\pp_2) +\,\cdots \, + {\rm mult}_I ( \pp_k ).
\end{equation}
If $I$ is a monomial ideal in $R = \KK[x_1,\ldots,x_n]$ then
${\rm amult}(I)$ is the number of {\em standard pairs}~$(\delta,\pp_i)$, by
\cite[Lemma 3.3]{STV_DEGREE}. In such a pair, $\pp_i$ is  a monomial prime and we can
identify $\delta$  with a differential operator $ \prod_{x_j \in \pp_i} \partial_{x_j}^{u_j}$.
The set $\AAA_i$ of such operators describes the contribution of the coordinate
subspace $V(\pp_i)$ to the scheme $V(I)$, and this yields the minimal representation \autoref{eq:diffprimdec}.

\begin{example} \label{ex:fromVogel}
Let $R = \mathbb{Q}[x,y,z]$ and $I = \langle x^2 y, x^2 z, xy^2, xyz^2 \rangle $
as in \cite[eqn (1.5)]{STV_DEGREE}. 
This ideal has  $k=4$ associated primes, namely
$\pp_1 = \langle x \rangle$, 
$\pp_2 = \langle y,z \rangle$,
$\pp_3 = \langle x,y \rangle$ and
$\pp_4 = \langle x,y,z \rangle$.
We see in \cite[eqn (3.3)]{STV_DEGREE} that membership in $I$ 
is characterized by  $\,{\rm amult}(I) = 5\,$ Noetherian operators:
$$ 
\AAA_1 =  \{ 1 \},\,
\AAA_2 =  \{ 1 \},\,
\AAA_3 =  \{ \partial_x \}\,\, {\rm and} \,\,\,
\AAA_4 =  \{  \partial_x \partial_y,\,\partial_x \partial_y \partial_z \}.
$$
\end{example}

Our contribution is a general theory that resolves both problems \hyperref[problemA]{(a)} and \hyperref[problemA]{(b)}. 
We propose a variant of \autoref{eq:diffprimdec}  where 
$\AAA_i$ consists of operators in ${\rm Diff}_{R / \KK}(R,R/\pp_i)$.
Such differential primary decompositions always exist, they satisfy $|\AAA_i| \geq {\rm mult}_I(\pp_i)$,
and equality is attained for all~$i$.

The presentation is organized as follows.
In \autoref{section_recap} we fix the set-up
and we review basics on differential operators
in commutative algebra. 
In \autoref{sect_diff_prim_dec_ideals}
we introduce differential primary decompositions. Our
main result on their existence and minimality appears in
 \autoref{thm:main}.
\autoref{sec:modules} generalizes
this result from ideals to modules.
We then specialize to formally smooth $\KK$-algebras $R$,
where differential operators in ${\rm Diff}_{R / \KK}(R,R)$ suffice and
\autoref{eq:diffprimdec} is valid as stated.

In  \autoref{sec:polynomials}
we turn to the case of polynomial rings, which is most relevant for applications.
\autoref{thm_poly_case} extends the results on primary ideals in \cite{PRIM_IDEALS_DIFF_EQS}.
We show how to compute  the pairs $(\pp_i,\AAA_i)$ for any ideal $I$ in
$R = \KK[x_1,\ldots,x_n]$ with $\text{char}(\KK)=0$. We discuss our
{\tt Macaulay2} implementation, we present non-trivial examples, and we
reflect on applications to  linear PDE as in \cite{Bjoerk, Ehrenpreis, Hoermander, PALAMODOV}.
Further practical  tools for
solving  such linear PDE can be found in \cite{aitelmanssour_harkonen_sturmfels_2021, HHS}.

\section{Differential Operators in Commutative Algebra}
\label{section_recap}

Throughout this paper we assume the setup below.
In this section, we fix notation and we recall some foundational results to be used
(for further details, the reader is referred to \cite[\S 16]{EGAIV_IV}).

\begin{setup}
	\label{setup_1}
	Let $\KK$ be a field and let $R$ be a $\KK$-algebra essentially of finite type over $\KK$.
	Let $A $ be a  $\KK$-subalgebra of $R$ such that $R$ is essentially of finite type over $A$.
	This means that $R$ is the localization of a finitely generated $A$-algebra.
	We are mostly interested in the case $A = \KK$.
\end{setup}

For any prime $\pp \in \Spec(R)$, we denote by $k(\pp)$ the residue field 
$k(\pp):=R_\pp/\pp R_\pp = \Quot(R/\pp)$.
For two $R$-modules $M$ and $N$, we regard $\Hom_A(M, N)$ as an $(R\otimes_A R)$-module, by setting 
$$
\left((r \otimes_A s) \delta\right)(w) \,=\, r \delta(sw) \quad \text{ for all } \delta \in \Hom_A(M, N), \; w \in M,\; r,s \in R. 
$$ 
We use the bracket notation $[\delta,r](w) = \delta(rw)-r\delta(w)$ for $\delta \in \Hom_A(M, N)$, $r \in R$ and $w \in M$.

Unless specified otherwise, whenever we consider an $(R \otimes_A R)$-module as an $R$-module, we do 
so by letting $R$ act via the left factor of $R \otimes_A R$. We now introduce our most relevant module.

\begin{definition}
	\label{def_diff_ops}
	Let $M, N$ be $R$-modules.
	The \textit{$m$-th order $A$-linear differential operators}, denoted
	$\,
	\Diff_{R/A}^m(M, N) \subseteq \Hom_A(M, N)$,
	form an $(R\otimes_A R)$-module that is defined inductively~by
	\begin{enumerate}[\rm (i)]
		\item $\Diff_{R/A}^{0}(M,N) \,:=\, \Hom_R(M,N)$.
		\item $\Diff_{R/A}^{m}(M, N) \,:= \,
		\big\lbrace \delta \in \Hom_A(M,N) \,\mid\, [\delta, r] \in \Diff_{R/A}^{m-1}(M, N) 
		\,\text{ for all }\, r \in R \big\rbrace$.
	\end{enumerate}
	The set of all \textit{$A$-linear differential operators from $M$ to $N$} 
	is the $(R \otimes_A R)$-module
	$$
	\Diff_{R/A}(M, N) \,\,:=\,\, \bigcup_{m=0}^\infty \Diff_{R/A}^m(M,N).
	$$
	Subsets $\mathcal{E} \subseteq \Diff_{R/A}(M, N)$ are viewed
	as differential equations. Their solution~spaces are
	\begin{equation}
		\label{eq:solE}
		{\rm Sol}(\mathcal{E}) \,\,:= \,\,\big\lbrace w \in M \,\mid \, \delta(w) = 0 \text{ for all } \delta \in \mathcal{E} \big\rbrace 
		\,\,= \,\,\bigcap_{\delta \in \mathcal{E} } {\rm Ker}(\delta).
	\end{equation}
\end{definition}

\begin{example}
If $R=\KK[x_1,\ldots,x_n]$ is a polynomial ring over a field $\KK$ of characteristic zero, then  
$\Diff_{R/\KK}(R,R)$ is the Weyl algebra $D_n = R\langle \partial_{x_1},\ldots, \partial_{x_n} \rangle$.
For a derivation see e.g.~\cite[Lemma 1]{PRIM_IDEALS_DIFF_EQS}.
\end{example}

To describe differential operators, one uses the module of principal parts.
Consider the multiplication map 
$
	\mu : R \otimes_A R \rightarrow R, \; 	r \otimes_A s \mapsto rs.
$
The kernel of this map is the ideal $\Delta_{R/A} \subset R \otimes_A R$. 
\begin{definition}
	Let $M$ be an $R$-module.
	The {\em module of $m$-th principal parts} is defined as	
	$$
	P_{R/A}^m(M) \,:=\, \frac{R \otimes_A M}{	\Delta_{R/A}^{m+1}  \left(R \otimes_A M\right)}.
	$$
This is a module over $R \otimes_A R$ and thus over $R$.
	For simplicity of notation, set $ P_{R/A}^m := P_{R/A}^m(R)$.
\end{definition}
For any $R$-module $M$, we consider the universal map 
$\, d^m : M \rightarrow P_{R/A}^m(M), \,\, w  \mapsto \overline{1 \otimes_A w} $.
The following result is a fundamental characterization of the modules of differential operators.

\begin{proposition}[{\cite[Proposition 16.8.4]{EGAIV_IV}, \cite[Theorem 2.2.6]{AFFINE_HOPF_I}}]
	\label{prop_represen_diff_opp}
	Let $M$ and $N$ be $R$-modules and let $m\ge 0$.
	Then, the following map is an isomorphism of $R$-modules:
	\begin{align*}
		{\left(d^m\right)}^*\, :\, \Hom_R\left(P_{R/A}^m(M), N\right) 
		&\,\,\xrightarrow{\cong} \,\,\Diff_{R/A}^m(M, N), \\
		\varphi  \quad &\,\,\mapsto \,\quad \varphi \circ d^m.
	\end{align*}
\end{proposition}

The next lemma records basic facts about the localization of differential operators. 

\begin{lemma} \label{lem_localization} Let $M,N$ be $R$-modules with $M$ finitely generated.
	Given $W \subset R$ multiplicatively closed, we set $S= W^{-1}R$, $M'=W^{-1}M$ and $N'=W^{-1}N$.
	Then the following statements hold:
	\begin{enumerate}[\rm (i)]
		\item There is a  canonical localization map 
		$$
		\Diff_{R/A}(M,N) \;\rightarrow\; S \otimes_R \Diff_{R/A}(M,N) \;\cong\; \Diff_{S/{(W \cap A)}^{-1}A }(M',N').
		$$
		\item Assume that $N$ is $W$-torsion-free.
Consider a subset $\AAA \subset \Diff_{R/A}(M,N)$ and let $\AAA'$ be the image of $\AAA$ under the map in part (i).
		Then, we have the equality $\,\Sol(\AAA) = \Sol(\AAA') \cap M$.
	\end{enumerate}	
\end{lemma}
\begin{proof}
	(i) See, e.g., \cite[Lemma 2.7]{NOETH_OPS}.
	
	(ii) The proof follows verbatim that of \cite[Proposition 3.6~(ii)]{NOETH_OPS}.
	Let $\delta \in \AAA$ and consider its extension $\delta' \in \AAA'$.
	We have the following commutative diagram  of $A$-linear maps:
	\begin{center}
		\begin{tikzpicture}[baseline=(current  bounding  box.center)]
\matrix (m) [matrix of math nodes,row sep=2em,column sep=6em,minimum width=2em, text height=1.3ex, text depth=0.25ex]
			{
				M & N \\
				M' & N' \\
			};
			\path[-stealth]
			(m-1-1) edge node [above] {$\delta$} (m-1-2)
			(m-2-1) edge node [above] {$\delta'$} (m-2-2)
			(m-1-1) edge (m-2-1)
			;		
			\draw[right hook->] (m-1-2)--(m-2-2);			
		\end{tikzpicture}	
	\end{center}
	The vertical map on the right is injective. 
	Hence $\Ker(\delta) = \Ker(\delta') \cap M$, and the claim follows.
\end{proof}

For the sake of completeness, we describe the localization of differential operators explicitly.

\begin{remark}
	\label{rem_explicit_localiz}
	Given $\delta \in \Diff_{R/A}^m(M, N)$, we extend it to an element $\delta' \in \Diff_{S/{(W \cap A)}^{-1}A}^m(M',N')$.
	We proceed by induction on $m$.
	If $m=0$, then $\delta \in \Hom_R(M,N)$ and $\delta'$ is defined by setting $\delta'(\frac{\beta}{w}) = \frac{\delta(\beta)}{w}$ for all $\beta \in M, w \in W$.	If $m > 0$, then we set
$$ \qquad \delta'\left(\frac{\beta}{w}\right) \,=\, \frac{\delta(\beta) - [\delta,w]'(\frac{\beta}{w})}{w}
\qquad \text{for all $\beta \in M$ and $w \in W$.}
$$
This is well-defined by the induction hypothesis and the fact that $[\delta,w] \in \Diff_{R/A}^{m-1}(M,N)$.
\end{remark}

Here is another basic remark regarding the solutions to differential operators of order $m$.

\begin{remark}
	\label{rem_incl_power_pp}
	Let $M$ be an $R$-module, $\pp \in \Spec(R)$ a prime ideal and $\delta \in \Diff_{R/A}^m(M, R/\pp)$.
	Then, by induction on $m$, it follows that $\Sol(\delta) \supseteq \pp^{m+1}M$.
	If $m=0$, then $\delta \in \Hom_R(M, R/\pp)$ and the result is clear. 
	If $m > 0$, for $r \in \pp$ and $\beta \in \pp^mM$, then we obtain the identity
	$$
	\delta(r\beta) \,=\, r\delta(\beta)  + [\delta,r](\beta)\, =\, 0.
	$$
	The latter equation follows from $[\delta,r] \in \Diff_{R/A}^{m-1}(M,R/\pp)$ and the induction hypothesis.
\end{remark}

The following lemma will be used for a process of lifting differential operators.

\begin{lemma}[{\cite[Proposition 3.12]{NOETH_OPS}}]
	\label{lem_lifting}
	Suppose that $R$ is formally smooth over $A$. 
	Let $\pp \in \Spec(R)$ be a prime ideal and $F$ a free $R$-module of finite rank.
	Then,  the canonical map 
	$\,	\Diff_{R/A}^m(F, R) \rightarrow \Diff_{R/A}^m(F,R/\pp)	\,$ 
	is surjective for all $m \ge 0$.
\end{lemma}

For ease of notation, we fix the following piece of data for the rest of this section. 

\begin{notation} Assume \autoref{setup_1} with $A = \KK$.
Let $\pp \in \Spec(R)$ with residue field $\FF = k(\pp)$.
Let $\MM$ be the kernel of the multiplication map $\FF \otimes_\KK R_\pp \rightarrow \FF$.
Then $\MM = \Diag \left(\FF \otimes_\KK R\right)$ is the extension of $\Diag \subset R \otimes_\KK R$ into  $\FF \otimes_\KK R$. The ideal $\MM$ is maximal with $\left(\FF \otimes_\KK R_\pp\right) / \MM \cong \FF$.
\end{notation}

The following lemma will be used in the proof of \autoref{thm:main}.

\begin{lemma}[{\cite[Lemma 3.8, Lemma 3.14]{NOETH_OPS}}]
	\label{lem_diff_ops_correspon_to_ideals}
	Let $M$ be a finitely generated $R$-module.
	\begin{enumerate}[\rm (i)]
		\item There is an isomorphism of $(R_\pp \otimes_\KK R_\pp)$-modules 
		$$
		\Diff_{R_\pp/\KK}^m(M_\pp,\FF) \, \cong \, \Hom_{\FF}\left(\frac{\FF \otimes_\KK M_\pp}{\MM^{m+1}\left(\FF \otimes_\KK M_\pp\right)}, \,\FF\right). $$
		\item The $R_\pp$-module $\Diff_{R_\pp/\KK}^m(M_\pp,\FF)$ is a finite-dimensional vector space over the field $\,\FF$.
		\item The isomorphism in (i) induces a bijection between $\MM$-primary submodules 
		$	\mathbb{V} \subset \FF \otimes_\KK M_\pp	$ 
		containing $\MM^{m+1}\!\left(\FF \otimes_\KK M_\pp\right)$ and  $(R_\pp \otimes_\KK R_\pp)$-submodules of 
		$
		\Diff_{R_\pp/\KK}^m(M_\pp,\FF)
		$. It is given by
		$$
		\MM^{m+1}\!\left(\FF \otimes_\KK M_\pp\right) \,\, \subset \,\,\,\mathbb{V}  \,
		\;\mapsto\; \Hom_{\FF}\left(\frac{\FF \otimes_\KK M_\pp}{\mathbb{V}}, \FF\right) \,\, =: \,\, \mathcal{E}.
		$$
		\item Under the bijection in (iii), viewing $\,\mathcal{E}$ as
		differential equations, the solution space in \autoref{eq:solE}~is
		$$ \Sol(\mathcal{E}) \,\,=\,\, \mathbb{V} \cap M_\pp, $$
		where $\mathbb{V} \cap M_\pp$ denotes the contraction of $\,\mathbb{V}$ under the canonical inclusion 
		$$ M_\pp \,\cong\, 1 \otimes_\KK M_\pp \,\hookrightarrow\, \FF \otimes_\KK M_\pp. $$
	\end{enumerate}
\end{lemma}

The lemma below will be used to reduce to a separable algebraic setting.

\begin{lemma}[{\cite[Proposition 3.6~(i)]{NOETH_OPS}}]
	\label{lem_extend_field}
	If $\KK$ is a perfect field, then there is an intermediate field $\KK \subset \mathbb{L} \subset R_\pp$ such that $\mathbb{L} \hookrightarrow \FF$ is a separable algebraic extension.
\end{lemma}

The next proposition contains a technical result of fundamental importance for our approach.

\begin{proposition}[{\cite[Proposition 4.1]{BRUMFIEL_DIFF_PRIM}, \cite[Proposition 3.9]{NOETH_OPS}}]
	\label{prop_ass_gr_rings}
	If $\KK \hookrightarrow \FF$ is a separable algebraic extension, then we have an isomorphism of local rings
	\begin{equation*}
 \varepsilon_{\pp,m}  \; :  \; R_\pp/\pp^{m}R_\pp \;\;\xrightarrow{\cong}\;\; \left(\FF \otimes_\KK R_\pp\right)/\MM^m \quad \text{ for all }\,  m \ge 1.
	\end{equation*}
\end{proposition}

\section{Arithmetic multiplicity and Noetherian operators}
\label{sect_diff_prim_dec_ideals}

In this section we introduce a notion of primary decomposition that is based on differential operators.
We use the notation and results in \autoref{section_recap}.
The setup below is set throughout.

\begin{setup}
	\label{setup_general}
	Let $\KK$ be a field and $R$ be a $\KK$-algebra essentially of finite type over $\KK$.
\end{setup}

Our definition rests on localizing along associated prime ideals $\pp_i$. This is an essential feature.

\begin{definition}
	\label{def_prim_dec}
	Let $I \subset R$ be an ideal with $\Ass(R/I) = \{\pp_1,\ldots,\pp_k\} \subset \Spec(R)$.
	A \emph{differential primary decomposition} of $I$ is a list of pairs $(\pp_1, \AAA_1), \ldots, (\pp_k, \AAA_k)$,
where $\AAA_i \subset \DiffR(R,R/\pp_i)$ is a finite set of differential operators,
such that the following equation holds for  each $\pp \in \Ass(R/I)$:
\begin{equation}
\label{eq:defprimdec}
I_\pp  \,\;=\; \bigcap_{1 \le i \le k \atop \pp_i \subseteq \pp}  \big\lbrace f \in R_\pp \mid \delta'(f) = 0 \text{ for  all }   \delta \in \AAA_i  \big\rbrace.
\end{equation}
Here	 $\delta' \in \Diff_{R_\pp/\KK}(R_\pp,R_\pp/\pp_iR_\pp)$ denotes the image of an
operator $\delta \in \AAA_i$ under  \autoref{lem_localization}~(i).
\end{definition}

The next lemma records fundamental properties of differential primary decompositions.

\begin{lemma}
	\label{lem_basic_prop_weak}
For a differential primary decomposition as in
\autoref{def_prim_dec}, the following holds:
	\begin{enumerate}[\rm (i)]
		\item The ideal $I$ satisfies \autoref{eq:diffprimdec}, i.e.~
		 $I = \big\lbrace f \in R \mid \delta(f) = 0 \text{ for  all } \,\delta \in \AAA_i  \text{ and }
		  \, 1 \le i \le k   \big\rbrace$.
		\item For each index $i \in \{1,\ldots,k\}$, the set   $\,\AAA_i$ of differential operators is non-empty.
	\end{enumerate}
\end{lemma}
\begin{proof}
Fix a minimal primary decomposition $I=Q_1 \cap \cdots \cap Q_k$ where $Q_i$ is a $\pp_i$-primary ideal.
	
(i)  By applying   \autoref{lem_localization} (ii) to the localization at $\pp  \in \Ass(R/I)$, we get 
	\begin{equation} \label{eq_intersect_pp}
		\bigcap_{i: \pp_i \subseteq \pp} Q_i \;=\; IR_\pp \cap R
  \quad =\; \bigcap_{i: \pp_i \subseteq \pp}  \big\lbrace f \in R \mid \delta(f) = 0 \text{ for  all }   \delta \in \AAA_i  \big\rbrace.
	\end{equation}
The equality on the left holds because
	 $Q_i$ is a $\pp_i$-primary ideal (see, e.g., \cite[Theorem 4.1]{MATSUMURA}).
	By intersecting the equation above over all $\pp \in \Ass(R/I)$, we obtain the 
	desired identity \autoref{eq:diffprimdec}.
	
	(ii) Suppose that $\AAA_i=\emptyset$ for some $i \in \{1,\ldots,k\}$, and set $\qqq = \pp_i$.
	By intersecting \autoref{eq_intersect_pp} for all $\pp \in \Ass(R/I)$ with $\pp \subsetneq \qqq$, we obtain
	$\,\bigcap_{j: \pp_j \subsetneq \qqq} Q_j  \;=\; \bigcap_{j: \pp_j \subsetneq \qqq}  
	\big\lbrace f \in R \mid \delta(f) = 0 \text{ for  all }   \delta \in \AAA_j  \big\rbrace $.
	Here each $\pp_j$ is strictly contained in $\qqq$.	
	Since $\AAA_i=\emptyset$, by definition,
	 we conclude $\bigcap_{j: \pp_j \subsetneq \qqq} Q_j  = IR_\qqq \cap R$.	
 This contradicts the hypothesis $\,\qqq \in \Ass(R/I)$, and so the result follows.
\end{proof}

\begin{remark} Differential primary decompositions always exist.
Let $I = Q_1  \cap \cdots \cap Q_k$ where each $Q_i$ is a $\pp_i$-primary ideal. 
By \cite[Theorem A]{NOETH_OPS}, there is a finite set $\BBB_i' \subset \Diff_{R_{\pp_i}/\KK}(R_{\pp_i},k(\pp_i))$ of Noetherian operators such that $Q_iR_{\pp_i} = \lbrace f \in R_{\pp_i} \mid \delta(f) = 0 \text{ for  all }   \delta \in \BBB_i'  \rbrace$.
By lifting $\BBB_i'$ to a set of operators $\BBB_i \subset \DiffR(R,R/\pp_i)$, it follows that
$(\pp_1, \BBB_1), \ldots, (\pp_k, \BBB_k)$ is a differential primary decomposition of $I$.
Hence the requirement in  \autoref{def_prim_dec} can always be achieved.
\end{remark}

What we are looking for here is more ambitious.
We wish to represent any ideal $I$ by a \emph{minimal differential primary decomposition}.
This should reflect the intrinsic complexity of the affine scheme defined by $I$.
This brings us back to the notion from  \cite{STV_DEGREE} we saw in the Introduction.

\begin{definition}
	\label{def_arith_mult}
		For an ideal $I \subset R$, its \emph{arithmetic multiplicity} is the positive integer 
$$ \amult(I) \,\,\, :=  \sum_{\pp \in \Ass(R/I)} \!\!\! \text{length}_{R_\pp}\left(\HH_{\pp}^0\left(R_\pp/IR_\pp\right)\right) \quad =
 \sum_{\pp \in \Ass(R/I)} \!\!\! \! \text{length}_{R_\pp}\left(\frac{\left(IR_\pp:_{R_\pp} {(\pp R_\pp)}^{\infty}\right)}{IR_\pp}		\right). $$
 In our Introduction and in \cite{STV_DEGREE},
 the length inside the sum was denoted ${\rm mult}_I(\pp)$ and called the
 multiplicity of $I$ along $\pp$. It is the length of the largest ideal of finite length in the ring
 $R_{\pp}/I R_{\pp}$.
\end{definition}

The next theorem is our main result in this section. 
An ideal $I$ always has a  differential primary decomposition whose total number of 
operators is equal to the arithmetic multiplicity. 
Moreover,  ${\rm amult}(I)$ is a lower bound on the size of any
differential primary decomposition.

\begin{theorem} \label{thm:main}	 
Assume \autoref{setup_general} with $\KK$ perfect. 
Fix an ideal $I \subset R$ with $\Ass(R/I) = \{\pp_1,\ldots,\pp_k\} \subset \Spec(R)$.
The size of a differential primary decomposition is at least $ {\rm amult}(I)$, and this
upper bound is tight. 
More precisely:
\begin{enumerate}[\rm (i)]
\item $I$ has a differential primary decomposition $(\pp_1, \AAA_1)$, $\ldots$, $(\pp_k, \AAA_k)$ such that 
$ |\AAA_i | = {\rm mult}_I(\pp_i)$.
\item If $(\pp_1, \AAA_1)$, $\ldots$, $(\pp_k, \AAA_k)$ is a differential primary decomposition for $I$, then 
$|\AAA_i | \, \ge  \,{\rm mult}_I(\pp_i)$.
\end{enumerate}		
\end{theorem}

The proof of  \autoref{thm:main} appears further below.
We start with a proposition that transfers the approximation technique
used in \cite{NOETH_OPS, PRIM_IDEALS_DIFF_EQS}
to ideals that are not necessarily primary. 

\begin{proposition} \label{prop_approx}
	Fix an ideal $I \subset R$ and an associated prime $\pp \in \Ass(R/I)$.
	Let $J = (I :_R \pp^\infty)$ and let $\,\FF = k(\pp)$ be the residue field of $\pp$.
	Assume that $\KK \hookrightarrow \FF$ is a separable algebraic extension.
	Then the following statements hold: 
	\begin{enumerate}[\rm (i)]
		\item There exists a positive integer $m_0$ such that, for all $m \ge m_0$, we have the isomorphism 
		$$
		\frac{J_\pp}{I_\pp} \; \, \xrightarrow{\cong} \,\; \frac{J_\pp + \pp^mR_\pp}{I_\pp + \pp^mR_\pp}.
		$$
		\item 
Using the canonical map 	$\,\gamma_{\pp} :  R_\pp \;\rightarrow \;\FF \otimes_\KK R_\pp$, we
	define the ideals
$$		 \aaa_m = \gamma_\pp(I_\pp) + \MM^m \quad \text{ and  } \quad \bbb_m = \gamma_\pp(J_\pp) + \MM^m .$$
		Then, if we choose $m \ge m_0$  as in part (i), we obtain the following isomorphism
		$$
		J_\pp/I_\pp  \; \xrightarrow{\cong} \; \bbb_m/\aaa_m.
		$$
		\item 
		Let $\mathcal{E}_m$ be the  $\,(R_\pp \otimes_\KK R_\pp)$-submodule of $\Diff_{R_\pp/\KK}^{m-1}\left(R_\pp,\FF\right)$ that is determined by  the $\MM$-primary ideal
		$\,\aaa_m $ as in \autoref{lem_diff_ops_correspon_to_ideals}~(iii).
		Then the localized ideal $I_\pp$ is recovered as follows:
		$$
		I_\pp  \, = \, \bigcap_{m=1}^\infty \Sol(\mathcal{E}_m).
		$$
	\end{enumerate}
\end{proposition}

\begin{proof}
	(i) We have a canonical surjection ${J_\pp} \twoheadrightarrow \frac{J_\pp + \pp^mR_\pp}{I_\pp + \pp^mR_\pp}$. 
	From this we obtain the isomorphism 
	$$
	\frac{J_\pp}{J_\pp \cap (I_\pp + \pp^mR_\pp)} \; \xrightarrow{\cong} \; \frac{J_\pp + \pp^mR_\pp}{I_\pp + \pp^mR_\pp}.
	$$
We must show that	$J_\pp \cap (I_\pp + \pp^mR_\pp) = I_\pp$ for $m \gg 0$.
The left hand side contains the right hand side for any $m \ge 0$.
	On the other hand, since $J_\pp \supset I_\pp$,  we have
	$\, J_\pp \cap (I_\pp + \pp^mR_\pp) = I_\pp + J_\pp \cap \pp^mR_\pp$.
	By localizing we get $J_\pp = (I_\pp :_{R_\pp} (\pp R_\pp)^\infty)$, and this implies $J_\pp \cap \pp^mR_\pp \subset I_\pp$ for $m \gg 0$.
	
	(ii) From \autoref{prop_ass_gr_rings} we have the isomorphisms  
	$$
	R_\pp/\left(I_\pp + \pp^mR_\pp\right) \xrightarrow{\cong} \left(\FF \otimes_\KK R_\pp\right)/\aaa_m \;\;\text{ and }\;\; R_\pp/\left(J_\pp + \pp^mR_\pp\right) \xrightarrow{\cong} \left(\FF \otimes_\KK R_\pp\right)/\bbb_m.
	$$
Hence the result is obtained by combining these isomorphisms with that in part~(i).
	
(iii) \autoref{lem_diff_ops_correspon_to_ideals}~(iv) and \autoref{prop_ass_gr_rings} give the equality 
$ \bigcap_{m=1}^\infty \! \Sol(\mathcal{E}_m) = \bigcap_{m=1}^\infty \! \left(I_\pp + \pp^mR_\pp\right)$.
Finally, by Krull's Intersection Theorem  \cite[Theorem 8.10]{MATSUMURA}, the right hand side equals $I_\pp$.	
\end{proof}

From now on we shall use the following notation.
If $W \subset R$ is a multiplicatively closed set and $\delta$ is a  differential operator, then $\mathfrak{L}_W(\delta)$ 
denotes the extension/localization of $\delta$ with respect to $W$ as in \autoref{lem_localization}.
For a prime ideal $\pp \in \Spec(R)$, we set $\mathfrak{L}_\pp(\delta) =  \mathfrak{L}_{R \setminus \pp}(\delta)$.

\begin{proof}[Proof of  \autoref{thm:main}]
We write  $I=Q_1 \cap \cdots \cap Q_k$, where $Q_i$ is a $\pp_i$-primary ideal, and
$\pp_i \neq \pp_j$ for $i \neq j$. We also assume that
the $k$ indices are ordered such that $\pp_j \subsetneq \pp_i$ implies $j < i$.
	
\smallskip
	
(i) We  proceed by induction. Fix $i \in \{1,\ldots, k\}$	and assume the following hypotheses to hold:
\begin{enumerate}[(a)]
\item there exist $\AAA_1,\ldots,\AAA_{i-1}$ with $\AAA_j \subset \DiffR(R,R/\pp_j)$ and $|\AAA_j| = 
{\rm mult}_I(\pp_j)$	for $1 \le j \le i-1$;
		\item \label{induct_hyp_b} for all $1 \le j \le i-1$, the following identity holds:
\begin{equation}
\label{eq:identityholds}
		\bigcap_{
				1 \le \ell \le j \atop
				\pp_\ell \subseteq \pp_j			
		} Q_{\ell}R_{\pp_j} \; = \;  
		\bigcap_{
				1 \le \ell \le j \atop
				\pp_\ell \subseteq \pp_j
		} 
		\big\lbrace f \in R_{\pp_j} \mid \mathfrak{L}_{\pp_j}(\delta)(f) = 0 \text{ for all } \delta \in \AAA_\ell \big\rbrace.
\end{equation}
	\end{enumerate}
	These hold vacuously for the base case $i=1$.	
To simplify notation, we set $\pp = \pp_i$ and $\FF=k(\pp)$. 

Our aim is to find  $\AAA = \AAA_i \subset \DiffR(R,R/\pp)$ such that $|\AAA | =
{\rm mult}_I(\pp)$ and \autoref{eq:identityholds} holds with $j=i$.
For each $\xi \in \Diff_{R_\pp/\KK}(R_\pp,\FF)$, \autoref{lem_localization} yields $\delta \in \DiffR(R,R/\pp)$ and 
$r \in R \backslash \pp$ such that $\xi = \frac{\mathfrak{L}_\pp(\delta)}{r}$.
So, we localize at $\pp$ and we  consider operators in $\Diff_{R_\pp/\KK}(R_\pp,\FF)$.
These can be lifted.
\autoref{lem_extend_field} gives a field extension $\KK \subset \LL \subset R_\pp$ such that 
 $\LL \hookrightarrow \FF$ is separable and algebraic.
 By \cite[Lemma 2.7~(ii)]{NOETH_OPS}, we have $\Diff_{R_\pp/\LL}(R_\pp,\FF) \subset  \Diff_{R_\pp/\KK}(R_\pp,\FF)$.
We now set $\KK = \LL$ and this makes \autoref{prop_approx} applicable.
Setting $J = (I:_R\pp^\infty)$, we have the primary decompositions 
$$
	I_\pp \,\,= \bigcap_{1 \le \ell \le i \atop \pp_\ell \subseteq \pp		
	} \!\! Q_{\ell} R_\pp \quad \text{ and } \quad 
	J_\pp \,\,=\!\! \bigcap_{1 \le \ell \le i-1 \atop\pp_\ell \subseteq \pp}\!\! Q_{\ell} R_\pp,
$$	
where $J_\pp = R_\pp$ if $\pp$ is a minimal prime of $I$.
	We now divide the proof into three shorter steps.
	
	{\sc Step 1.}
Let $\aaa_m = \gamma_\pp(I_\pp) + \MM^m$  and  $\bbb_m = \gamma_\pp(J_\pp) + \MM^m$
as in  \autoref{prop_approx}~(ii).
Following \autoref{lem_diff_ops_correspon_to_ideals}~(iii), let $\mathcal{E}_m $
be the $(R_\pp \otimes_\KK R_\pp)$-submodule determined by the inclusion
	$$
	\mathcal{E}_m \cong \Hom_{\FF}\left(
	\frac{\FF \otimes_\KK R_\pp}{\aaa_m}, \FF \right) \,\hookrightarrow \,
	 \Hom_{\FF}\left(\frac{\FF \otimes_\KK R_\pp}{\MM^{m}}, \FF\right) \cong 
	\Diff_{R_\pp/\KK}^{m-1}(R_\pp,\FF). 
	$$ 
 	Since $\, I_\pp   =  \bigcap_{m=1}^\infty \Sol(\mathcal{E}_m)$, by
   \autoref{prop_approx}~(iii), we  restrict ourselves to studying the $\mathcal{E}_m$'s.
 	
 	{\sc Step 2.}
	The idea is to ``delete or not take into account'' the differential conditions 
	for describing $J_\pp$ that are available by induction.
 	 	Let $m \ge 0$. We have the short exact sequence 
 	$$
 	0 \rightarrow \frac{\bbb_m}{\aaa_m} \rightarrow  \frac{\FF \otimes_\KK R_\pp}{\aaa_m}  \rightarrow \frac{\FF \otimes_\KK R_\pp}{\bbb_m} \rightarrow 0.
 	$$
	By dualizing with the functor $\Hom_{\FF}(-,\FF)$, we obtain the short exact sequence 
 	\begin{equation}
 		\label{eq_short_exact}
 0 \,\rightarrow \,\Hom_{\FF}\left(\frac{\FF \otimes_\KK R_\pp}{\bbb_m}, \FF\right)\, 
 \rightarrow \, \mathcal{E}_m \, \rightarrow \,\Hom_{\FF}\left(\frac{\bbb_m}{\aaa_m},\FF\right) \,\rightarrow \,0.
 	\end{equation}
 	Notice that $\Hom_{\FF}\left(\frac{\FF \otimes_\KK R_\pp}{\bbb_m}, \FF\right)$ is isomorphic to an $(R_\pp \otimes_\KK R_\pp)$-submodule $\mathcal{H}_m \subset \Diff_{R_\pp/\KK}^{m-1}(R_\pp,\FF) $ such that $\Sol(\mathcal{H}_m) = J_\pp + \pp^mR_\pp$. Again, this follows from \autoref{lem_diff_ops_correspon_to_ideals} and \autoref{prop_ass_gr_rings}.

 	By \autoref{eq_short_exact}, we can find an $\FF$-basis 
	 $B_m = \{\omega_1^{(m)},\ldots,\omega_{r_m}^{(m)}, \xi_1^{(m)},\ldots,\xi_{s_m}^{(m)}\}$
	 of the $\FF$-vector space $\mathcal{E}_m$ such that $\{w_1^{(m)},\ldots,\omega_{r_m}^{(m)}\}$  is an $\FF$-basis of the subspace $\mathcal{H}_m \subset \mathcal{E}_m$ and 
 	$s_m = \dim_{\FF}\left(\bbb_m/\aaa_m\right)$.
 	
 	 {\sc Step 3.} Although the problem now seems to be of an infinite nature, it is not:
	 we can bound the order $m$.
 	 By \autoref{prop_approx}~(i),(ii), there exists $m_0 \ge 1$ such that
	 $ J_\pp/I_\pp  \, \xrightarrow{\cong} \, \bbb_m/\aaa_m$
 	 is an isomorphism for $m \ge m_0$.
	Hence, for $m \ge m_0$, we obtain the following commutative diagram:
	\begin{center}
		\begin{tikzpicture}[baseline=(current  bounding  box.center)]
			\matrix (m) [matrix of math nodes,row sep=2.5em,column sep=4em,minimum width=2em, text height=2ex, text depth=0.25ex]
			{
				0 & \mathcal{H}_m & \mathcal{E}_m & \Hom_{\FF}\left(\frac{\bbb_m}{\aaa_m},\FF\right) & 0\\
				0 & \mathcal{H}_{m+1} & \mathcal{E}_{m+1} & \Hom_{\FF}\left(\frac{\bbb_{m+1}}{\aaa_{m+1}},\FF\right) & 0\\
			};
			\path[-stealth]
			(m-1-1) edge (m-1-2)
			(m-1-2) edge (m-1-3)
			(m-1-3) edge (m-1-4)
			(m-1-4) edge (m-1-5)
			(m-2-1) edge (m-2-2)
			(m-2-2) edge (m-2-3)
			(m-2-3) edge (m-2-4)
			(m-2-4) edge (m-2-5)
			(m-1-4) edge node [right] {$\cong$} (m-2-4)
			;		
			\draw[right hook->] (m-1-2)--(m-2-2);					
			\draw[right hook->] (m-1-3)--(m-2-3);
		\end{tikzpicture}	
	\end{center}
The rows are exact, the two left vertical maps are inclusions, and the right one is an isomorphism.
 We choose the $\FF$-bases $B_m$ and $B_{m+1}$ above in such a way that $\{\xi_1^{(m)},\ldots,
 \xi_{s_m}^{(m)}\} = \{\xi_1^{(m+1)}\! \!,\ldots,\xi_{s_{m+1}}^{(m+1)}\}$. This set stabilizes and 
 we denote it by $\{\xi_1, \ldots,\xi_s\} \subset \mathcal{E}_{m_0}$. Its cardinality~is
$$ s \,=\, \dim_{\FF}\left(\bbb_{m_0}/\aaa_{m_0}\right) \,=\,  \leng_{R_\pp}\left(J_\pp/I_\pp\right) \,=\, {\rm mult}_I(\pp). $$
	
Using Krull's Intersection Theorem and the induction hypothesis \hyperref[induct_hyp_b]{(b)}, we find
 	 \begin{equation}
 	 	\label{eq_J_pp_as_sols}
 	 	\bigcap_{m=1}^{\infty} \! \Sol(\mathcal{H}_m) \,=\, \bigcap_{m=1}^{\infty} \!(J_\pp + \pp^mR_\pp) \,=\, \,J_\pp\,
		 \; = \bigcap_{1 \le \ell \le i-1 \atop\pp_\ell \subseteq \pp} \!\!
 	 	\big\lbrace f \in R_\pp \mid \mathfrak{L}_\pp(\delta)(f) = 0 \text{ for all } \delta \in \AAA_\ell \big\rbrace.		
 	 \end{equation}
	 Setting $\AAA'  = \{\xi_1, \ldots,\xi_s\}$ and using
	  the way the $\FF$-bases $B_m$ were chosen, we conclude
 	 \begin{align}
 	 	\label{eq_large_sol}
 	 	\begin{split} 	 	
 	 	I_\pp \,\,&= \bigcap_{m=m_0}^\infty \!\!\Sol(\mathcal{E}_m) \,\,
 = \bigcap_{m=m_0}^\infty \!\!\Sol(B_m) \,\,= \,\Big(\bigcap_{m=m_0}^\infty \Sol(\mathcal{H}_m)\Big) \,\,\bigcap\,\, \Sol(\AAA')\\
 &= \bigcap_{1 \le \ell \le i-1 \atop \pp_\ell \subseteq \pp } \!\!
 \big\lbrace f \in R_\pp \mid \mathfrak{L}_\pp(\delta)(f) = 0 \text{ for all } \delta \in \AAA_\ell \big\rbrace \,\,\bigcap\,\, \Sol(\AAA').
 	 \end{split}
 	 \end{align}
 So, we obtained the identity that proves part (i).

\medskip
  	
 (ii)  Consider any differential primary decomposition	 $(\pp_1, \AAA_1), \ldots, (\pp_k, \AAA_k)$ for the given ideal~$I$.
  	Fix $1 \le i \le k$, and set $\pp = \pp_i$, $\FF = k(\pp)$, $\AAA = \AAA_i$ and $J = (I :_R \pp^\infty)$.
  	By assumption, we have
  	\begin{align}
  		\label{eq_intersect_lower_bound}
		\begin{split}
I_\pp  \,&\,\,=\,\,\, J_\pp \; \, \bigcap \; \,	\big\lbrace f \in R_\pp \mid \mathfrak{L}_\pp(\delta)(f) = 0 \text{ for all } \delta \in \AAA \big\rbrace.
		\end{split}
  	\end{align}
The map $\iota : J_\pp \hookrightarrow R_\pp$ induces a canonical map 
 $  \tau :
 \Diff_{R_\pp/\KK}(R_\pp, \FF) \,\rightarrow \,\Diff_{R_\pp/\KK}(J_\pp, \FF),  \delta \mapsto \delta \circ \iota.
 $
 Since $I_\pp$ is a $\pp$-primary submodule of $J_\pp$, we
 can assume  $\AAA \subset \Diff_{R_\pp/\KK}^m(R_\pp, \FF)$ and $\pp^{m+1}J_\pp \subseteq I_\pp$.
 Let $\widetilde{\AAA} \subset \Diff_{R_\pp/\KK}^m(J_\pp, \FF)$ be the image of
  $\{ \mathfrak{L}_\pp(\delta) \mid \delta \in \AAA\}$ under $\tau$. It follows from
\autoref{eq_intersect_lower_bound} that
	\begin{equation}
		\label{eq_Ipp_as_as_equat_Jpp}
I_\pp \,\,=\,\, \big\lbrace f \in J_\pp \mid \widetilde{\delta}(f) = 0 \text{ for all } \widetilde{\delta} \in \widetilde{\AAA} \big\rbrace.
	\end{equation}
Applying   $-\otimes_{R_\pp} J_\pp$ to  $R_\pp/\pp^{m+1}R_\pp \rightarrow \left(\FF \otimes_\KK R_\pp\right)/\MM^{m+1}$, we obtain the map
	\begin{equation}
		\label{eq_isom_Jpp_m}
  J_\pp/\pp^{m+1}J_\pp \;\cong\; R_\pp/\pp^{m+1}R_\pp \otimes_{R_\pp} J_\pp \,\;\rightarrow\;\,
   \frac{\FF \otimes_\KK R_\pp}{\MM^{m+1}} \otimes_{R_\pp} J_\pp 
		\;\cong\; \frac{\FF \otimes_\KK J_\pp}{\MM^{m+1}\left(\FF \otimes_\KK J_\pp\right)} 
		=: \mathcal{Q}.
	\end{equation}
The module $\mathcal{Q}$ appears in \autoref{lem_diff_ops_correspon_to_ideals}~(i). 
Let $\mathcal{G} \subset \Hom_{\FF}\left(\mathcal{Q},\FF\right)$ be the $(R_\pp \otimes_\KK R_\pp)$-module generated by $\widetilde{\AAA}$ in
$\Hom_{\FF}\left(\mathcal{Q},\FF\right)$.
Note that $\mathcal{G}$ is a finitely generated module over the Artinian local ring $\left(\FF \otimes_\KK R_\pp\right)/\MM^{m+1}$.
Let  $\mathbb{V} \subset \FF \otimes_\KK J_\pp$ such that 
	\begin{equation}
		\label{eq_functionals_sols}
		\frac{\mathbb{V}}{\MM^{m+1}\left(\FF \otimes_\KK J_\pp\right)} \;\cong \; \big\lbrace w \in \mathcal{Q} \,\mid\, \eta(w) = 0 \;\text{ for all }\; \eta \in \mathcal{G}  \big\rbrace.
	\end{equation}
	Dualizing the inclusion $\mathcal{G} \subset \Hom_{\FF}\left(\mathcal{Q},\FF\right)$, we get the short exact sequence 
	\begin{equation} \label{eq_dualize_short_E}
		0 \,\,\rightarrow \,\,Z \,\,\rightarrow\,\, \mathcal{Q}
		\,\, \rightarrow \,\,\Hom_\FF(\mathcal{G}, \FF)\,\, \rightarrow \,\, 0,
	\end{equation}
	where $Z=\big\lbrace w \in \mathcal{Q}
	\,\mid\,  \eta(w) = 0 \;\text{ for all }\; \eta \in \mathcal{G} \big\rbrace$.
	We get the isomorphism 
$\Hom_\FF(\mathcal{G}, \FF) \cong \frac{\FF \otimes_\KK J_\pp}{\mathbb{V}}$
by combining \autoref{eq_functionals_sols} and \autoref{eq_dualize_short_E}.
By \autoref{eq_Ipp_as_as_equat_Jpp} and \autoref{eq_isom_Jpp_m}, we get an inclusion $J_\pp/I_\pp  \hookrightarrow \frac{\FF \otimes_\KK J_\pp}{\mathbb{V}}$.
This implies	$$
|\AAA| = |\widetilde{\AAA}| \,\ge \,\dim_\FF\left(\Hom_\FF(\mathcal{G}, \FF)\right) 
\,=\, \dim_{\FF}\left(\frac{\FF \otimes_\KK J_\pp}{\mathbb{V}}\right) \,\ge\, \leng_{R_\pp}(J_\pp/I_\pp) \,=\,
{\rm mult}_I(\pp).	$$
This is the desired inequality, which
 completes the proof of \autoref{thm:main}.
\end{proof}

\begin{remark}
After substituting $\KK$ by an intermediate field $\KK \subset \LL \subset R_\pp$ such that $\LL \hookrightarrow \FF$ is separable and algebraic, the minimal differential primary decomposition arose from a 
compatible basis for the $\mathbb{F}$-vector spaces
$\mathcal{H}_m \subset \mathcal{E}_m$ in \autoref{eq_short_exact}. 
 Every such basis gives a set of operators
for $\mathfrak{A}_i$.
In that sense, the minimal differential primary decomposition is unique up to choices of bases. 
\end{remark}

\section{From Ideals to Modules}
\label{sec:modules}

We now turn to differential primary decompositions for modules. 
The methods to be used for modules are the same as for ideals.
We continue using the setup and notation in \autoref{sect_diff_prim_dec_ideals}.
We fix a finitely generated $R$-module $M$. In applications,
this is typically a free module $M = R^p$.

\begin{definition} \label{def_prim_dec_modules}
Let $U \subset M$ be an $R$-submodule with $\Ass(M/U) = \{\pp_1,\ldots,\pp_k\} \subset \Spec(R)$.
 A \emph{differential primary decomposition} of $U$ is a list of pairs $(\pp_1, \AAA_1), \ldots, (\pp_k, \AAA_k)$ such that $\AAA_i \subset \DiffR(M,R/\pp_i)$ is a finite set of differential operators, and 
for each  $\pp \in \Ass(M/U)$ we have
$$ U_\pp  \,\,\,\;= \bigcap_{1 \le i \le k \atop \pp_i \subseteq \pp}  \big\lbrace w \in M_\pp \mid \delta'(w) = 0 \text{ for  all }   \delta \in \AAA_i  \big\rbrace.
$$				
Here	$\delta' \in \Diff_{R_\pp/\KK}(M_\pp,R_\pp/\pp_iR_\pp)$ denotes the image of an
operator $\delta \in \AAA_i$ under  \autoref{lem_localization}~(i).
\end{definition}

We begin with the analogs to \autoref{lem_basic_prop_weak}, \autoref{def_arith_mult} and \autoref{prop_approx}.
	
\begin{lemma}
For a differential primary decomposition as in \autoref{def_prim_dec_modules}, the following holds:	
		\begin{enumerate}[\rm (i)]
\item The submodule is recovered as
 $\,U = \big\lbrace w \in M \mid \delta(w) = 0 \text{ for  all }  \delta \in \AAA_i  \text{ and } 1 \le i \le k  \big\rbrace$.
 \item For each index $i \in \{1,\ldots,k\}$, the set   $\,\AAA_i$ of differential operators is non-empty.
		\end{enumerate}
	\end{lemma}
	\begin{proof}
Essentially verbatim to the proof of \autoref{lem_basic_prop_weak}.
	\end{proof}

\begin{definition}
	For a submodule $U \subset M$, its \emph{arithmetic multiplicity} 
is the positive integer
$$ \amult(U) \,\, :=  \sum_{\pp \in \Ass(M/U)} \!\!\!
\text{length}_{R_\pp}\left(\HH_{\pp}^0\left(M_\pp/U_\pp\right)\right) 
\quad = \sum_{\pp \in \Ass(R/I)} \!\!\! \text{length}_{R_\pp}\left(
\frac{\left(U_\pp:_{M_\pp} {(\pp R_\pp)}^{\infty}\right)}{U_\pp} \right).
$$
\end{definition}

\begin{example} 
Primary decompositions for submodules have been studied in
computer algebra (cf.~\cite{Indrees}), but explicit examples are rare,
even over a polynomial ring.
 Working with their differential operators is unfamiliar, and the development of
 numerical algorithms is highly desirable. 
Let us consider $R = \mathbb{Q}[x,y,z]$,
$M = R^2$ and 
$ U = {\rm image}_R \! \begin{small} \begin{bmatrix} x^2 & xy & xz \\
y^2 & yz & z^2 \end{bmatrix}\end{small} = U_1 \cap U_2 \cap U_3$, where
$$ U_1 \,=\, {\rm image}_R  \begin{bmatrix} 0 & x \\ 1 & 0 \end{bmatrix} ,\quad
U_2 \,=\, {\rm image}_R   \begin{bmatrix} x & y^2 & 0 \\ z & z^2 & xz- y^2 \end{bmatrix}, \quad
U_3 \,=\,  {\rm image}_R   \begin{bmatrix} 
	1 & 0 & 0 & 0\\ 
	0 & y^2 & yz  & z^2
\end{bmatrix}.
$$
The modules $U_i$ are primary with associated primes
$\pp_1 = \langle x \rangle $,
$\pp_2 = \langle xz - y^2\rangle $,
$\pp_3 = \langle y,z \rangle$.
Each prime has multiplicity one in $U$, so  ${\rm amult}(U) = 3$.
A  minimal differential primary decomposition  is given by
$U = \{w \in R^2 \mid \,\delta_i (w) \in \pp_i \,\,{\rm for}\,\, i=1,2,3 \}$ where
$\delta_1 = (1,0)$, $\delta_2 = (z,-x) $, $\delta_3 = (0,\partial_z) $. 
We verified this example with the homological methods described in 
\cite[Section 1]{EHV}.
\end{example}

The next proposition is our approximation result for the case of modules.

\begin{proposition} \label{prop_approx_mod}
	Fix an $R$-submodule $U \subset M$ and an associated prime $\pp \in \Ass(M/U)$.
	Let $V = (U :_M \pp^\infty)$ and $\FF = k(\pp)$ be the residue field of $\pp$.
	Assume that $\KK \hookrightarrow \FF$ is a separable algebraic extension.
	Then the following statements hold: 
\begin{enumerate}[\rm (i)]
\item There exists a positive integer $m_0$ such that, for all $m \geq m_0$, we	 have the isomorphism 
$$
\frac{V_\pp}{U_\pp} \; \, \xrightarrow{\cong} \; \, \frac{V_\pp + \pp^mM_\pp}{U_\pp + \pp^mM_\pp}.
$$
\item 
Let $\gamma_{\pp,M}  : M_\pp  \rightarrow  \FF \otimes_\KK M_\pp$ be the map induced by $\gamma_\pp$
in  \autoref{prop_approx} (ii), and define the $(R_\pp \otimes_\KK R_\pp)$-modules
$$
\aaa_m = \gamma_{\pp,M}(U_\pp) + \MM^m\left(\FF \otimes_\KK M_\pp\right) \quad \text{ and  } \quad \bbb_m = \gamma_{\pp,M}(V_\pp) + \MM^m\left(\FF \otimes_\KK M_\pp\right). $$
Then, if we choose $m \ge m_0$  as in part (i), we obtain the following isomorphism
$$ V_\pp/U_\pp  \; \xrightarrow{\cong} \; \bbb_m/\aaa_m. $$
\item Let $\mathcal{E}_m$ be the  $(R_\pp \otimes_\KK R_\pp)$-submodule of $\Diff_{R_\pp/\KK}^{m-1}\left(M_\pp,\FF\right)$  determined by  the $\MM$-primary submodule
		$\,\aaa_m $ as in \autoref{lem_diff_ops_correspon_to_ideals}~(iii).
		 The localized module $U_\pp$ is recovered as~follows:
		$$
		U_\pp  \,\, = \, \bigcap_{m=1}^\infty \Sol(\mathcal{E}_m).
		$$		
	\end{enumerate}
\end{proposition}

\begin{proof}
Part (i) is obtained identically to \autoref{prop_approx}~(i), and similarly for part (iii).
For part (ii) we use 	\autoref{prop_ass_gr_rings}.
Taking the tensor product $- \otimes_{R_\pp} M_\pp$, we obtain the isomorphism
 $$
	\frac{M_\pp}{\pp^{m}M_\pp} \;\;\cong\;\; \frac{R_\pp}{\pp^{m}R_\pp} \otimes_{R_\pp} M_\pp \;\;\xrightarrow{\cong}\;\; \frac{\FF \otimes_\KK R_\pp}{\MM^m} \otimes_{R_\pp} M_\pp \;\;\cong\;\; \frac{\FF \otimes_\KK M_\pp}{\MM^{m}\left(\FF \otimes_\KK M_\pp\right)}
\qquad \text{for all $m \ge 1$.} $$
 From this we get the isomorphisms:
 	$$
 	M_\pp/\left(U_\pp + \pp^mM_\pp\right) \xrightarrow{\cong} \left(\FF \otimes_\KK M_\pp\right)/\aaa_m \;\;\text{ and }\;\; M_\pp/\left(V_\pp + \pp^mM_\pp\right) \xrightarrow{\cong} \left(\FF \otimes_\KK M_\pp\right)/\bbb_m.
 	$$
 	So, the proof is analogous to that of \autoref{prop_approx}~(ii).
 	\end{proof}

The next theorem is an extension of our main result (\autoref{thm:main}) to the case of modules.

\begin{theorem}
	\label{thm_modules}
	Assume \autoref{setup_general} with $\KK$ perfect.
For any submodule $U \subset M$ with $\Ass(M/U) = \{\pp_1,\ldots,\pp_k\} \subset \Spec(R)$, we have:
	\begin{enumerate}[\rm (i)]
		\item $U$ has a differential primary decomposition $(\pp_1, \AAA_1)$, $\ldots$, $(\pp_k, \AAA_k)$ such that 
$$ |\AAA_i | \, = \, \leng_{R_{\pp_i}}\left(\HH_{\pp_i}^0\left(M_{\pp_i}/U_{\pp_i}\right)\right). $$			
\item If $(\pp_1, \AAA_1),\ldots, (\pp_k, \AAA_k)$ is any differential primary decomposition for $U$, then $$ |\AAA_i | \, \ge  \, \leng_{R_{\pp_i}}\left(\HH_{\pp_i}^0\left(M_{\pp_i}/U_{\pp_i}\right)\right). $$
Thus,  the size of a differential primary decomposition is at least	$\,\amult(U)$.
\end{enumerate}		
\end{theorem}

\begin{proof}
Fix a primary decomposition $U = N_1 \cap \cdots \cap N_k$ where $N_i \subset M$ is a $\pp_i$-primary 
submodule of $M$. Without any loss of generality, we order the primary submodules $N_1,\ldots,N_k$ 
in such a way that $\pp_j \subsetneq \pp_i$ implies $j < i$. For the proof of (i) 
we proceed as in \autoref{thm:main}~(i). Namely, we
	 use induction on $i=1,2,\ldots,k$ to derive the 
	 representation by differential operators:
	$$
	\bigcap_{
		1 \le \ell \le i \atop
		\pp_\ell \subseteq \pp_i	}  \!\!
	\left(N_{\ell}\right)_{\pp_i} \; = \;  
	\bigcap_{1 \le \ell \le i \atop \pp_\ell \subseteq \pp_i} \!\!\!
	\big\lbrace w \in M_{\pp_i} \mid \mathfrak{L}_{\pp_i}(\delta)(w) = 0 \text{ for all } \delta \in \AAA_\ell \big\rbrace.
	$$
Here $\AAA_i \subset \DiffR(M,R/{\pp_i})$
	is carefully constructed to satisfy $|\AAA_i |  =  \leng_{R_\pp}\left(\HH_{\pp}^0\left(M_\pp/U_\pp\right)\right)$.
	The main change is that we now use \autoref{prop_approx_mod} instead of \autoref{prop_approx}.

The proof of the lower bound in part (ii)	
	mirrors that of \autoref{thm:main}~(ii).
\end{proof}

In  \autoref{thm:main} and \autoref{thm_modules} we used differential
operators that take values in $R/\pp_i$ rather than in $R$.  This was necessary
for the existence of a differential primary decomposition. Indeed, the 
representation \autoref{eq:diffprimdec} is generally not available
for differential operators that map into~$R$.

One way to remedy this is to restrict the class of rings $R$. To this end,
we now assume that the $\KK$-algebra $R$ is {\em formally smooth} over 
the ground field $\KK$, and we work in a free module $M = R^p$.
For the definition and basic properties of
formally smooth algebras we refer to
 \cite[\href{https://stacks.math.columbia.edu/tag/00TH}{Tag 00TH}]{stacks-project}.

We are interested in a notion of differential primary decomposition that utilizes
operators in $\DiffR(R,R)$ and $\DiffR(M,R)$ respectively. Consider an ideal $I \subset R$
or a submodule $U \subset M$. The representation 
in  \autoref{def_prim_dec} is called a \emph{strong differential primary decomposition} of $I$
if we can choose $\AAA_i \subset \DiffR(R,R)$ for $i=1,\ldots,k$.
The representation 
in \autoref{def_prim_dec_modules} is called a \emph{strong differential primary decomposition} of $M$
if $\AAA_i \subset \DiffR(M,R)$ for $i=1,\ldots,k$.

The following important result follows as a corollary from our previous developments.

\begin{corollary} \label{cor:formallysmooth}
Assume \autoref{setup_general} with $\KK$ perfect.
Let $R$ be formally smooth over $\KK$ and $M = R^p$ a free module of finite rank.
For any submodule $U \subset M$ with $\Ass(M/U) = \{\pp_1,\ldots,\pp_k\} \subset \Spec(R)$,
we have:
	\begin{enumerate}[\rm (i)]
		\item $U$ has a strong differential primary decomposition $(\pp_1, \AAA_1)$, $\ldots$, $(\pp_k, \AAA_k)$ such that 
		$$
		|\AAA_i | \, = \, \leng_{R_{\pp_i}}\left(\HH_{\pp_i}^0\left(M_{\pp_i}/U_{\pp_i}\right)\right).
		$$			
		\item If $(\pp_1, \AAA_1)$, $\ldots$, $(\pp_k, \AAA_k)$ is a strong differential primary decomposition for $U$, then $$
		|\AAA_i | \, \ge  \, \leng_{R_{\pp_i}}\left(\HH_{\pp_i}^0\left(M_{\pp_i}/U_{\pp_i}\right)\right).
		$$
		Thus,  the size of a strong differential primary decomposition is at least
		$
		\amult(U).
		$
	\end{enumerate}		
\end{corollary} 

\begin{proof}
	This follows directly from  \autoref{lem_lifting} and \autoref{thm_modules}.
\end{proof}

The next example shows that the hypothesis of $\KK$ being perfect  cannot be avoided.
The minimality result in \autoref{thm:main}~(i) and \autoref{thm_modules}~(i) 
may fail without that assumption.

\begin{example}
Fix a prime $p \in \NN$ and the rational function field  $\KK=\mathbb{F}_p(t)$ over
  $\mathbb{F}_p = \ZZ/p\ZZ$.   Consider the
  maximal ideal   $\pp =  \langle x^p - t \rangle $ in the polynomial ring $R = \KK[x]$.	It is
  well known that $\DiffR(R, R) = \bigoplus_{m=0}^\infty R D_x^m$ where $D_x^m = \partial_x^m/m!\,$ 
  is the differential operator determined by 
	$$
	D_x^m \bigl(x^\beta \bigr) \,=\, \binom{\beta}{m}x^{\beta-m} \;\; \text{ for all }\;\; \beta \ge 0.
	$$ 
	The multiplicity of $ \pp^2$  is $\text{mult}_{\pp^2}(\pp) = 2$.
	However, the minimal number of Noetherian operators required to describe the $\pp$-primary ideal $\pp^2$ is equal to $p+1$.
	An explicit minimal description is 
	$$
	\pp^2 \;=\; \big\lbrace f \in R \mid D_x^m(f) \in \pp \;\text{ for all }\; 0 \le m \le p \big\rbrace.
	$$
The right hand side is a $\pp$-primary ideal  by  \cite[Proposition 3.5]{NOETH_OPS}. It must equal $\pp^2$ since
the only $\pp$-primary ideals in $R$ are powers of $\pp$, and $D_x^m(\pp) \subset \pp$ for $m \leq p-1$.
	See also \cite[Example 5.3]{NOETH_OPS}.
\end{example}

We close with noting
that a differential primary decomposition in a polynomial ring can be pushed forward to a quotient algebra. 
The result below is proved by the methods in \autoref{section_recap}.

\begin{proposition}
	\label{lem_push_forward}
	Let $S=\KK[x_1,\ldots,x_n]$ and $R = S / \mathfrak{K}$ for an ideal $\mathfrak{K} \subset S$.
	Given an ideal $J \subset S$ with $J \supseteq \mathfrak{K}$, it represents an ideal
	 $I = \overline{J} = J / \mathfrak{K} $ in $R$. 
	\begin{enumerate}[\rm (i)]
	\item We have a canonical inclusion $\DiffR^m(R, R/I) \hookrightarrow \Diff_{S/\KK}^{m}(S,S/J)$ for all $m \ge 0$.
	\item If $\delta \in \Diff_{S/\KK}^m(S,S/J)$ and $\Ker(\delta) \supset \mathfrak{K}$, then
		$\delta$ is in the image of the inclusion in part {\rm (i)}.
	\end{enumerate}
\end{proposition}

The promised pushforward works as follows.
This is specially important from a practical point of view, since
 \autoref{algo} below is restricted to  polynomial rings $S$ with $\text{char}(\KK)=0$.

\begin{corollary}
	\label{cor_push_forward}
		Let $(\pp_1,\AAA_1),\ldots,(\pp_k,\AAA_k)$ be a differential primary decomposition for $J$
		as in \autoref{lem_push_forward}.
	 Then $\AAA_i \subset \Diff_{S/\KK}(S,S/\pp_i)$ can be identified with a set of differential operators $\widetilde{\AAA_i} \subset \DiffR(R,R/\overline{\pp_i})$, and $(\overline{\pp_1},\widetilde{\AAA_1}),\ldots,(\overline{\pp_k},\widetilde{\AAA_k})$ is a differential primary decomposition for $I$.
\end{corollary}
\begin{proof}
Pick $m$ such that $\AAA_i \subset \Diff_{S/\KK}^m(S,S/\pp_i)$.
	We have $\pp_i \supset \mathfrak{K}$ and $\Ker(\delta) \supset \mathfrak{K}$ for all $\delta \in \AAA_i$.
By applying
	  \autoref{lem_push_forward} with $\pp_i$ and $\overline{\pp_i}$, we can take
	$  \widetilde{\AAA_i} $ to be the preimage of $\AAA_i$ under the canonical inclusion $\DiffR^m(R, R/\overline{\pp_i}) \hookrightarrow \Diff_{S/\KK}^{m}(S,S/\pp_i)$.
This implies the assertion.
\end{proof}

We conclude with an illustration for the non-smooth ring  \autoref{eq:Rnonsmooth} discussed in the Introduction.

\begin{example}
	Let $S = \mathbb{Q}[x,y,z]$, $R = S /\langle x^3+y^3+z^3 \rangle$, $I = 
	\langle \overline{x}^{2}\overline{z},\overline{y}^{3}+\overline{z}^{3},\overline{x}^{2}\overline{y}
	\rangle \subset R$ and $J=\left(x^{2}z,y^{3}+z^{3},x^{2}y,x^{3}+y^{3}+z^{3}\right) \subset S$.
	Minimal primary decompositions for the two ideals~are 
\begin{equation}
\label{eq:Q1Q2Q3}
	J \,=\, Q_1 \cap Q_2 \cap Q_3 \;\quad \text{ and } \;\quad I \,=\, 
	\overline{Q_1} \cap \overline{Q_2} \cap \overline{Q_3},
\end{equation}	
	where $Q_1=\langle y+z,x^{2}\rangle$, $Q_2 = 
	\langle y^{2}-y\,z+z^{2},x^{2} \rangle $ and $Q_3 = \langle y+z,z^{2},x^{2}z,x^{3}\rangle$.
	Their radicals are
	$\pp_1 = \langle y+z,x \rangle$, $\pp_2 = \langle x,y^{2}-y\,z+z^{2} \rangle$ and 
	$\pp_3 = \langle x,y,z\rangle$, with
	$\text{mult}_J(\pp_1) = 2$, $\text{mult}_J(\pp_2) = 2$ and $\text{mult}_J(\pp_3) = 1$. 
Note that $Q_2$  and $\pp_2$ would break into two components over $\mathbb{C}$,
the setting in  \autoref{eq:Rnonsmooth}, but we here use $\mathbb{Q}$.
	A minimal strong differential primary decomposition for $J$ in $S$ equals
	\begin{equation}
		\label{eq_diff_prim_examp_J}
		\big( \pp_1,  \{1, \partial_x\} \big), \quad \big( \pp_2, \{1, \partial_x\} \big) \quad \text{and} \quad \big( \pp_3,  \{\partial_x^2\} \big).
	\end{equation}
		By \autoref{cor_push_forward}, we can 
		interpret \autoref{eq_diff_prim_examp_J} as a differential primary decomposition for $I$.
However, there is no strong differential primary decomposition for the ideal $I$ in the non-regular ring $R$.
To be precise, using the method in \cite[Example 5.2]{NOETH_OPS}, it can be shown that the contribution of the $\overline{\pp_3}$-primary ideal $\overline{Q_3}$ cannot be described by using differential operators in $\DiffR(R,R)$.
\end{example}

\section{Polynomial Rings}
\label{sec:polynomials}

In this section, we fix a field $\KK$ of characteristic zero and $R = \KK[x_1,\ldots,x_n]$.
Here, \autoref{cor:formallysmooth} holds. Differential operators
live in the Weyl algebra
$ {\rm Diff}_{R/\KK}(R,R)\,= \, \KK\langle x_1,\ldots,x_n,\partial_1,\ldots,\partial_n \rangle$.
For a free module $M = R^p$ we have
$ {\rm Diff}_{R/\KK}(M,R) =
 {\rm Diff}_{R/\KK}(R,R)^p  = 
 \KK\langle x_1,\ldots,x_n,\partial_1,\ldots,\partial_n \rangle^p$.
  In words, a differential operator on $M = R^p$ is a $p$-tuple of elements in the Weyl algebra. 
  In what follows we focus on explicit descriptions and computations. We will employ a framework similar to that
   in \cite{PRIM_IDEALS_DIFF_EQS}. One goal is to present
an algorithm which we implemented in \texttt{Macaulay2}~\cite{MACAULAY2}. 

\begin{remark} \label{rem:m2m2m2}
We begin with a convenient formula for the multiplicity of an associated prime:
 $$ \quad {\rm mult}_{\tt I}({\tt P}) \, = \, {\tt degree(saturate(I,P)/I)/degree(P)} $$
 This means that the arithmetic multiplicity $ {\rm amult}({\tt I}) $ can be computed 
 in  \texttt{Macaulay2} as follows:
 \begin{verbatim}      sum(apply( ass(I), P -> degree(saturate(I,P)/I)/degree(P) )) \end{verbatim}
As an illustration, we list the four associated primes in  \autoref{ex:fromVogel} with their multiplicities:
\begin{verbatim}
      R = QQ[x,y,z]; I = ideal(x^2*y, x^2*z, x*y^2, x*y*z^2);
      apply( ass(I), P -> {P,degree(saturate(I,P)/I)/degree(P)})
\end{verbatim}
The output shows that  the primes $\pp_1,\pp_2,\pp_3,\pp_4$ have multiplicities $1,1,1,2$
in this affine scheme. 
\end{remark}

We now turn to differential primary decompositions. 
Let $\mathcal{S} = \{x_{i_1},  \ldots, x_{i_\ell} \} $ be a subset of the variables in $R = \KK[x_1,\ldots,x_n]$.
Consider the polynomial subring $\KK[\mathcal{S}] :=\KK[x_{i_1}, \ldots,x_{i_\ell}] \subseteq R$ and
 the field of rational functions $\KK(\mathcal{S}) := \KK(x_{i_1},\ldots,x_{i_\ell})$.
	As in \cite{PRIM_IDEALS_DIFF_EQS}, we work with the  \emph{relative Weyl algebra},
	which is defined as the ring of $\KK[\mathcal{S}]$-linear differential operators on $R$:
	$$
	D_{n}(\mathcal{S}) \,:=\, \Diff_{R/\KK[\mathcal{S}]}(R,R) \,=\, R \big< \partial_{x_i} \mid x_i \not\in \mathcal{S} \big> \,\subseteq\,  R\langle \partial_{x_1},\ldots, \partial_{x_n} \rangle.
	$$
	Every differential operator $\delta \in D_{n}(\mathcal{S})$ is a unique $\KK$-linear combination  of standard monomials  $\,\mathbf{x}^\alpha \partial_\mathbf{x}^\beta =  x_1^{\alpha_1} \cdots x_n^{\alpha_n}  \left(\prod_{i \not\in \mathcal{S}} \partial_{x_i}^{\beta_i}\right)$, where $\alpha_i \in \mathbb{N}$, $\beta_i \in \NN$.
Differential operators $\delta \in D_n(\mathcal{S})$ act on
	polynomials $f  \in R$ in the familiar way, which is given by
	$\,  x_i ( f) = x_i \cdot f \,\, {\rm and} \,\, \partial_{x_i} ( f) = \partial f / \partial x_i $.

Let $\pp \in \Spec(R)$ be a prime ideal of dimension $\dim(R/\pp) = d$.
We say that  the set  of variables $\mathcal{S}$ is a {\em basis} modulo $\pp$ if
$|\mathcal{S}| = d$ and $\KK[\mathcal{S}] \cap \pp  = \{0\} $.
This specifies the  bases of the {\em algebraic matroid} of the prime $\pp$.
In the notation of  \cite[Example 13.2]{Mateusz}, this is the  algebraic
matroid associated with the generators $E = \{\overline{x}_1,\ldots,\overline{x}_n\}$
of the field extension $ K = k(\pp)$ over $F = \KK$.

We propose the following more refined notion of strong differential primary decomposition.
This following definition for polynomial rings differs from Definitions \ref{def_prim_dec} 
and \ref{def_prim_dec_modules} in that
every component is now a triplet, with the new entry being a choice of subset
of $\{x_1,\ldots,x_n\}$.

\begin{definition}
	\label{def_prim_def_poly}
	Let $I $ be an ideal  in the polynomial ring $R = \KK[x_1,\ldots,x_n]$,
	with associated primes $\Ass(R/I) = \{\pp_1,\ldots,\pp_k\}$.
	A \emph{differential primary decomposition} of $I$ is a list of triplets 
$$ (\pp_1, \mathcal{S}_1,\AAA_1), \;\; (\pp_2, \mathcal{S}_2,\AAA_2), \;\; \ldots \,, \;\; (\pp_k, \mathcal{S}_k,\AAA_k), $$
where  $\mathcal{S}_i$ is a basis modulo $\pp_i$ and
$\AAA_i $ is a finite set of differential operators in $ D_{n}(\mathcal{S}_i)$ such that
	$$ \qquad
	I_\pp \cap R \,\,\,=\; \bigcap_{i: \pp_i \subseteq \pp} \!\!
	 \big\lbrace f \in R \mid \delta ( f) \in \pp_i \text{ for  all }   \delta \in \AAA_i  \big\rbrace \qquad
\text{for each $\pp \in \Ass(R/I)$.} $$
This condition implies  \autoref{eq:diffprimdec}, so we get the desired
test of membership in $I$ by differential operators.
	\end{definition}		

What follows is our main result on
differential primary decompositions over polynomial rings. 

\begin{theorem} \label{thm_poly_case}
Fix a polynomial ideal $I \subset R$ with $\Ass(R/I) = \{\pp_1,\ldots,\pp_k\} \subset \Spec(R)$.
The size of a differential primary decomposition satisfies $k \geq {\rm amult}(I)$, and
 this is tight. More precisely,
\begin{enumerate}[\rm (i)]
\item $I$ has a differential primary decomposition $\{(\pp_i,  \mathcal{S}_i,\AAA_i)\}_{i=1,\ldots,k}$ such that
$ |\AAA_i | = {\rm mult}_I(\pp_i)$.
\item If  $\{(\pp_i,  \mathcal{S}_i,\AAA_i)\}_{i=1,\ldots,k}$ is any differential primary decomposition for $I$, then
$|\AAA_i | \, \ge  \,{\rm mult}_I(\pp_i)$.
\end{enumerate}
\end{theorem}

\begin{proof} As before, we fix a primary decomposition $I=Q_1 \cap \cdots \cap Q_k$ 
 where ${\rm rad}(Q_i) = \pp_i$ for $i=1,\ldots,k$. We assume that $\pp_j \subsetneq \pp_i$ implies $j < i$.
Let $\mathcal{S}_i \subset \{x_1,\ldots,x_n\}$ be a basis modulo $\pp_i$.	
	
(i) We proceed by induction. By \autoref{lem_lifting}, we may consider differential operators in $\Diff_{R/\KK[\mathcal{S}_i]}(R,R/\pp_i)$.
	Fix $i \in \{1,\ldots,k\}$ and assume that the following induction hypotheses~hold:
	\begin{enumerate}[\rm (a)]
		\item There exist $\AAA_1,\ldots,\AAA_{i-1}$ with $\AAA_j \subset D_n(\mathcal{S}_j)$ and 
		$|\AAA_j| = {\rm mult}_I(\pp_i)$
		  for all $1 \le j \le i-1$.
		\item  The identity \autoref{eq:identityholds} holds for $1 \le j \le i-1$,
with 	$\AAA_\ell$ identified with its image in $ \Diff_{R/\KK[\mathcal{S}_i]}(R,R/\pp_\ell)$.
	\end{enumerate}

	Set $\pp = \pp_i$, $\FF = k(\pp)$ and $\mathcal{S}=\mathcal{S}_i$.	
We invoke the same steps as in the proof of \autoref{thm:main}~(i) and arrive again at the conclusions of \autoref{eq_large_sol}.
We now work over the field $\KK(\mathcal{S})$ instead of $\KK$, and we construct  a subset 
$\BBB = \{\xi_1, \ldots,\xi_s\} \subset \Diff_{R_\pp/\KK(\mathcal{S})}(R_\pp,\FF)$ with $s = {\rm mult}_I(\pp)$ such that
$$ I_\pp \,\,\,	= \bigcap_{1 \le \ell \le i-1 \atop \pp_\ell \subseteq \pp} \!
\big\lbrace f \in R_\pp \mid \mathfrak{L}_\pp(\delta)(f) = 0 \text{ for all } 
\delta \in \AAA_\ell \big\rbrace \,\,\bigcap\,\, \Sol(\BBB). $$
Since $\mathcal{S}$ is a basis modulo $\pp$, we have
$R/\pp \otimes_{\KK[\mathcal{S}]}  \KK(\mathcal{S}) = \FF$, and we obtain canonical isomorphisms 
	\begin{equation}
		\label{eq_isom_diff_indep_set}	\Diff_{R_\pp/\KK(\mathcal{S})}(R_\pp,\FF) \cong \Diff_{\left(R\otimes_{\KK[\mathcal{S}]}  \KK(\mathcal{S})\right)/\KK(\mathcal{S})}(R\otimes_{\KK[\mathcal{S}]}  \KK(\mathcal{S}),\FF) \cong  \KK(\mathcal{S}) \otimes_{\KK[\mathcal{S}]} \Diff_{R/\KK[\mathcal{S}]}(R,R/\pp).
	\end{equation}
See \cite[Lemma 2.7~(iii)]{NOETH_OPS} for the isomorphism on the left.
For each $1 \le h \le s$, we now write  $\xi_h$	as $\frac{\delta_h}{r_h}$,
 where $\delta_h \in \Diff_{R/\KK[\mathcal{S}]}(R,R/\pp)$ and $r_h \in \KK[\mathcal{S}] \backslash \{0\}$.
	After lifting $ \{\delta_1, \ldots,\delta_s\}$ into $ \AAA_i \subset D_n(\mathcal{S}_i)$, we obtain
	the desired decomposition for~$I_\pp$.
		
	\medskip
	
	(ii) Suppose that $(\pp_1, \mathcal{S}_1,\AAA_1)$, $\ldots$, $(\pp_k, \mathcal{S}_k, \AAA_k)$ is a differential primary decomposition for $I$.
	Fix $\pp = \pp_i$ with $1 \le i \le k$, and set $\FF = k(\pp)$, $\AAA = \AAA_i$, $\mathcal{S}=\mathcal{S}_i$, $W = \KK[\mathcal{S}] \setminus \{0\}$, $S = W^{-1}R$, $\mathcal{I} = I_\pp \cap R$ and $J = (\mathcal{I} :_{R} \pp_i^\infty)$.
From  \autoref{def_prim_def_poly} and by analogy with \autoref{eq_Ipp_as_as_equat_Jpp}, we obtain 
	\begin{equation*}
\mathcal{I} \,\,=\,\, \big\lbrace f \in J \mid \widetilde{\delta}(f) = 0 
\text{ for all } \widetilde{\delta} \in \widetilde{\AAA} \big\rbrace.
	\end{equation*}
Here $\widetilde{\AAA}$ denotes the image of $\AAA$ under the canonical map 
	$$
	D_n(\mathcal{S}) = \Diff_{R/\KK[\mathcal{S}]}(R, R) \rightarrow \Diff_{R/\KK[\mathcal{S}]}(J, R/\pp), \;\; \delta \mapsto  \pi \circ \delta \circ \iota
	$$ 
	determined by $\iota : J \hookrightarrow R$ and $\pi : R \twoheadrightarrow R/\pp$.
	Since the operators in
	$\AAA \subset D_n(\mathcal{S})$ are $\KK[\mathcal{S}]$-linear,
	$$
	\mathcal{I}S \;=\; \big\lbrace f \in JS \mid \mathfrak{L}_W (\widetilde{\delta})(f) = 0 \text{ for all } \widetilde{\delta} \in \widetilde{\AAA} \big\rbrace.
	$$
Using the isomorphism in \autoref{eq_isom_diff_indep_set}, we can proceed as in
\autoref{thm:main}~(ii) to infer  $|\AAA| \ge  {\rm mult}_I(\pp)$.
\end{proof}

Finally, we present our algorithm for computing a minimal differential primary decomposition.
The algorithm is correct because it
realizes the steps in the proof of \autoref{thm_poly_case}~(i).
We use the representation in \cite[Theorem 2.1]{PRIM_IDEALS_DIFF_EQS} and the method for Noetherian operators in \cite[Algorithm~8.1]{PRIM_IDEALS_DIFF_EQS}.
 
 {
 	\onehalfspacing
\begin{algorithm}[Differential primary decomposition for an ideal in a polynomial ring]\label{algo}
	\hfill\\
	{\sc Input:} An ideal $I $ in $R = \KK[x_1,\ldots,x_n]$, where $ {\rm char}(\KK) = 0$. \\
	{\sc Output:}	A differential primary decomposition for $I$ of minimal size ${\rm amult}(I)$.
	
	\begin{enumerate}[(1)]
		\item Compute the set of associated prime ideals, $\Ass(R/I) = \{\pp_1,\ldots,\pp_k\}$.
		\item For $i$ from $1$ to $k$ do:
		\begin{enumerate}[(2.1)]
			\item Compute a basis $\mathcal{S}_i$ modulo $\pp_i$, and let $\FF_i = k(\pp_i)$ be the residue field of $\pp_i$.
			\item Compute the ideal $\mathcal{I} = I_{\pp_i} \cap R$ -- this is the intersection of
			all primary components of $I$ whose radical is contained in $\pp_i$.
			\item Compute the ideal  $J = \mathcal{I} :_R \pp_i^\infty $ -- this is the intersection of
			all  primary components of $I$ whose radical is strictly contained in $\pp_i$.
			\item Find $m >0$ giving the isomorphism in	 \autoref{prop_approx}~(i):
			$J/\mathcal{I}  \xrightarrow{\cong}  (J + {\pp_i}^m)/(\mathcal{I} + {\pp_i}^m)$.
						\item  By using \cite[Theorem 2.1, Algorithm 8.1]{PRIM_IDEALS_DIFF_EQS}, compute the $\FF_i$-vector subspaces $\mathcal{E}$ and $\mathcal{H}$ of the Weyl-Noether module 
$\FF_i \otimes_R D_n(\mathcal{S}_i) \,\cong\, \FF_i\big\langle \partial_{x_j} \mid x_j \not\in \mathcal{S}_i \big\rangle$.
These are $(R \otimes_{\KK[\mathcal{S}_i]} R)$-modules that correspond to the $\pp_i$-primary ideals $\mathcal{I} + \pp_i^m$ and $J + \pp_i^m$ respectively.
			\item Compute an $\FF_i$-vector subspace  $\mathcal{G} \subset \mathcal{E}$ complementary to $\mathcal{H}$ in $\mathcal{E}$, i.e., we have the direct sum $\mathcal{E} = \mathcal{H} \oplus \mathcal{G}$.
			Compute an $\FF_i$-basis $\overline{\AAA_i}$ of $\mathcal{G}$.
			\item Lift the basis $\overline{\AAA_i}$ to a subset $\AAA_i \subset D_n(\mathcal{S}_i)$ in the corresponding relative Weyl algebra.
		\end{enumerate}
		\item 
		Output the triples $(\pp_1, \mathcal{S}_1,\AAA_1)$, $\ldots$, $(\pp_k, \mathcal{S}_k, \AAA_k)$.
	\end{enumerate}
\end{algorithm}
}

We implemented this in {\tt Macaulay2} \cite{MACAULAY2}.
Our code is made available at {\url{mathrepo.mis.mpg.de}}.
This augments the package for primary ideals that is described in \cite{CCHKL},
and which rests on \cite{CHKL, PRIM_IDEALS_DIFF_EQS}.

The two commands in our implementation are called 
{\tt solvePDE} and {\tt getPDE}. The command {\tt solvePDE}
takes as its input an ideal $I$ in a polynomial ring and it 
creates a list of pairs $(\pp_i,\AAA_i)$ such that \autoref{eq:diffprimdec} holds and 
$|\AAA_i| = {\rm mult}_{\pp_i}(I)$ for all $i$. The command {\tt getPDE} reverses that process.
It starts from a list of Noetherian operators and computes ideal generators.
That reverse process does not check whether the given differential operators 
satisfy the conditions stipulated in \cite[Theorem 3.1]{PRIM_IDEALS_DIFF_EQS}.
Thus, running {\tt getPDE} after {\tt solvePDE} always returns the
ideal one starts with, but running {\tt solvePDE} after {\tt getPDE} might
lead to larger sets $\AAA_i$ than those one starts with.

\begin{example}
We run our two {\tt Macaulay2} commands on the ideal given in  \autoref{ex:fromVogel}:
\begin{verbatim}
load "noetherianOperatorsCode.m2"
R = QQ[x,y,z]; I = ideal(x^2*y,x^2*z,x*y^2,x*y*z^2)
solvePDE(I)
getPDE(oo)
\end{verbatim}
The output of the command {\tt solvePDE}(I) is the list of four pairs $(\pp_i, \AAA_i)$
that realizes \autoref{eq:diffprimdec}:
\begin{small}
\begin{verbatim}{{ideal x,{1}}, {ideal(y,z),{1}}, {ideal(x,y),{dx}},  {ideal(x,y,z),{dx*dy,dx*dy*dz}}}
\end{verbatim}
\end{small}
\end{example}

Our choice of the name {\tt solvePDE} is a reference to
the dual interpretation of the ideal $I$, namely as
a system of linear partial differential equations with 
constant coefficients. The Noetherian operators in 
$\AAA_i$ can be interpreted as polynomials
in $2n$ variables, called {\em Noetherian multipliers},
as in \cite[eqn (20)]{PRIM_IDEALS_DIFF_EQS}.
With this reinterpretation, our theory describes a
minimal integral representation for
all solutions to the given PDE. The command {\tt solvePDE} 
computes all solutions to  the PDE in the sense of the
Ehrenpreis-Palamodov Fundamental Principle 
\cite[Theorem 3.3]{PRIM_IDEALS_DIFF_EQS}.
We illustrate this for the binomial ideal  in \cite[Example 5.1]{DES},
which served in a statistics application.

\begin{example}
Consider the following system of linear PDE for an unknown function $f : \mathbb{R}^4 \rightarrow \mathbb{R}$:
\begin{equation} \label{eq:statPDE}
\frac{\partial^5 f}{\partial x_1^3 \partial x_3^2} =  \frac{\partial^5 f}{\partial x_2^5}\,,\,\,\,
\frac{\partial^5 f}{\partial x_2^2 \partial x_4^3} =  \frac{\partial^5 f}{\partial x_3^5}\, , \,\,\,
\frac{\partial^7 f}{\partial x_1^5 \partial x_4^2} =  \frac{\partial^7 f}{\partial x_2^7}\, , \,\,\,
\frac{\partial^7 f}{\partial x_1^2 \partial x_4^5} =  \frac{\partial^7 f}{\partial x_3^7}.
\end{equation}
We wish to describe all sufficiently differentiable functions $f$ that satisfy these four PDE.
The system \autoref{eq:statPDE} corresponds to an ideal $I$ in $R = \mathbb{Q}[x_1,x_2,x_3,x_4]$.
We enter this into {\tt Macaulay2}:
 \begin{verbatim}
  R = QQ[x1,x2,x3,x4];
  I = ideal( x1^3*x3^2-x2^5, x2^2*x4^3-x3^5, x1^5*x4^2-x2^7, x1^2*x4^5-x3^7 );
\end{verbatim}
This has four associated primes. The one minimal prime is the toric ideal
$\pp_1 = {\tt I}:\langle x_1x_2x_3 x_4 \rangle^\infty$. The three embedded primes are 
$\pp_2 = \langle x_1,x_2,x_3 \rangle $, 
$\pp_3 = \langle x_2,x_3,x_4 \rangle $, 
$\pp_4 = \langle x_1,x_2,x_3,x_4 \rangle $.
Using {\tt primaryDecomposition(I)}, we obtain a primary decomposition,
where the primary components have multiplicities
$1,67,60,916$. The methods in \cite{CCHKL, CHKL, PRIM_IDEALS_DIFF_EQS}
would compute $1044$ Noetherian operators to describe $I$. 
However,  \autoref{rem:m2m2m2} reveals that
${\rm amult}({\tt I}) = 1+18+18+170$ suffice.
Our command {\tt solvePDE(I)} computes
a minimal list of $170$ Noetherian operators in under ten minutes.
For instance, the last of the $18$ Noetherian operators $\delta$ displayed for the prime $\pp_2$ is
\begin{verbatim}
   x4^5*dx1*dx2*dx3^8 + 1120*x4^2*dx1*dx2^3*dx3^3 + 6720*dx1^3*dx2*dx3
\end{verbatim}
This translates into the following integral representation for certain special solutions to \autoref{eq:statPDE}:
$$ f(x_1,x_2,x_3,x_4) \,\,=\,\, \int \bigl(x_1 x_2 x_3^8\, t^5 + 1120 x_1x_2^3 x_3^3 \,t^2 + 6720 x_1^3 x_2 x_3\bigr) d \mu(t)  . $$
Here the notation is as in \cite[Theorem 3.3]{PRIM_IDEALS_DIFF_EQS}.
Our {\tt Macaulay2} output furnishes $170$ such formulas.
The (differential) primary decomposition 
gives insight into connectivity of random walks in \cite{DES}.
 \end{example}

The command {\tt solvePDE} can be used to compute solutions
for arbitrary homogeneous linear PDE with constant coefficients.
We believe that our results offer a recipe for
putting the Ehrenpreis-Palamodov theory from \cite{Bjoerk, Ehrenpreis, Hoermander, PALAMODOV}
into real-world practise. 
One crucial ingredient for this endeavor will be the development of numerical methods.
The advantage of numerical algorithms over symbolic ones was highlighted by Chen et al.~\cite{CHKL}, and we strongly agree with their assessment. A natural next step is the development
of an efficient numerical method whose input is a list of polynomials and whose output is 
a minimal differential primary decomposition. Likewise, it would be desirable to develop
a numerical algorithm for {\em primary fusion}, whose input consists of two ideals $I$ and $J$
and whose output is the intersection $I \cap J$. Here, each of the three ideals is encoded
by a minimal differential primary decomposition \autoref{eq:diffprimdec}. 
Primary fusion will describe the scheme-theoretic union of two affine schemes
in terms of differential operators.
Further developments in the context of applications were obtained in the subsequent work \cite{aitelmanssour_harkonen_sturmfels_2021}.


\end{document}